\documentclass[11pt,reqno,centertags]{amsart}
\usepackage{asymptote}
\usepackage{cite}
\usepackage{hyperref}
\hypersetup{
    colorlinks,
    citecolor=blue,
    filecolor=blue,
    linkcolor=blue,
    urlcolor=blue
}
\usepackage[english]{babel}
\usepackage{amsmath, amsfonts,amssymb}
\usepackage{enumerate}
\usepackage{graphicx}
\newcommand{\eps}{\varepsilon}
\newcommand{\nl}{\mathcal S_-}
\newcommand{\n}{\mathcal S}
\newcommand{\np}{\mathcal S_+}
\renewcommand{\v}{\mathcal{N}}
\newcommand{\vl}{\mathcal{N_-}}
\newcommand{\vp}{\mathcal{N_+}}
\renewcommand{\l}{\mathcal{W}}
\renewcommand{\ln}{\mathcal{W_-}}
\newcommand{\lv}{\mathcal{W_+}}
\newcommand{\p}{\mathcal{E}}
\newcommand{\pn}{\mathcal{E_-}}
\newcommand{\pv}{\mathcal{E_+}}
\newcommand{\nz}{n\mathbb{Z}^2}

\newcommand{\m}{P}

\newcommand{\zz}{\mathbb{Z}^2}
\newcommand{\rr}{\mathbb{R}^2}
\newcommand{\real}{\mathbb{R}}
\newcommand{\z}{\mathbb{Z}}

\newcommand{\set}[2]{\{#1 \colon #2 \}}

\newcommand{\diag}{\mathop{\mathrm{diag}}}

\newcommand{\floor}[1]{\lfloor #1\rfloor}
\newcommand{\Floor}[1]{\left\lfloor #1\right\rfloor}
\newcommand{\ceil}[1]{\lceil #1\rceil}
\newcommand{\Ceil}[1]{\left\lceil #1\right\rceil}

\usepackage{amsthm}
\newtheorem{theorem}{Theorem}[section]
\newtheorem{lemma}[theorem]{Lemma}
\newtheorem{proposition}[theorem]{Proposition}
\newtheorem{corollary}[theorem]{Corollary}
\newtheorem*{maintheorem}{Main Theorem}
\newtheorem*{subtheorem-a}{Sub-Theorem A}
\newtheorem*{subtheorem-b}{Sub-Theorem B}
\newtheorem*{subtheorem-c}{Sub-Theorem C}
\theoremstyle{definition}
\newtheorem{definition}[theorem]{Definition}
\newtheorem{example}[theorem]{Example}
\theoremstyle{remark}
\newtheorem{remark}[theorem]{Remark}

\numberwithin{equation}{section}
\begin{document}

\title{Existence of sublattice points in lattice polygons}
\author[N.~Bliznyakov]{Nikolai Bliznyakov}
\address{Faculty of Mathematics, Voronezh State University, 1 Universitetskaya 
pl., Voronezh, 394006, Russia}{}
\email{bliznyakov@vsu.ru}
\author[S.~Kondratyev]{Stanislav Kondratyev}
\address[S.~Kondratyev]{CMUC, Department of
Mathematics, University of Coimbra, 3001-501 Coimbra, Portugal}{}
\email{kondratyev@mat.uc.pt}




\begin{abstract}
We state the formula for the critical number of vertices of a convex lattice 
polygon that guarantees that the polygon contains at least one point of a 
given sublattice and give a partial proof of the formula.  We show that the 
proof can be reduced to finding upper bounds on the number of vertices in 
certain classes of polygons.  To obtain these bounds, we establish 
inequalities relating the number of edges of a broken line and the coordinates 
of its endpoints within a suitable class of broken lines.
\end{abstract}


\maketitle

\section{Introduction}

The study of lattice point in convex sets is a classical subject.  The 
starting point was Minkowski's Convex Body Theorem, which became the 
foundation of the geometry of numbers.  The theorem states that if a compact 
set in $\mathbb R^d$ is symmetric with respect to origin and has volume at 
least $2^d$, then it contains a point of the integer lattice~$\z^d$.  Notably, 
the constant~$2^d$ cannot be improved.  This theorem has quite a few 
modificatioins and generalisations, see e.~g.\ the nice short 
survey~\cite{Sco88}.  

There are numerous results concerning lattice points in various regions, see 
e.~g.~\cite{Cas12, GL87, EGH89, BR09}.  The regions at issue can be either 
general convex and nonconvex sets or polyhedra.  Among more recent works we 
note the following that are close to ours.  The papers \cite{Lov89, BCCZ10, 
MD11, Ave13} deal with the largest possible number of facets of maximal 
lattice-free polytopes.  The papers \cite{Hen83, LZ91, Pik01, AKN15} study 
properties of lattice polytopes having a specified (positive) number of 
interior lattice points such as upper bounds for the volume and the number of 
sublattice points and a classification of such polytopes.  The 
papers~\cite{Rab89, Rab90} deal with similar issues for polygons.

Besides, there are other interesting results about lattice polygons, such as 
\cite{Sco88, AZ95, Sim90}, not to mention the well-known Pick's theorem.

In this paper we consider the natural problem of relating the existence of 
sublattice points in a convex lattice polygon to the number of vertices (or 
edges) of the polygon.

In higher dimensions, a large number of faces cannot guarantee that the 
polytope will contain a point of a given sublattice.  For instance, there is 
no upper bound for the number of vertices and facets of polytopes in~$\mathbb 
R^3$ free of points of $(2\z)^3$.

Surprisingly, things are different in two dimensions.  It was noticed 
in~\cite{Bli00} that any convex integer pentagon on the plain contains a point 
of the lattice $(2\z)^2$.  In this paper we show that any convex integer 
polygon with many enough vertices contains at least one point of a given 
sublattice (of maximal rank) of~$\zz$.

In the spirit of the Minkowski Convex Body Theorem, our main goal is to state 
an explicit formula for the critical number of vertices ensuring that the 
polygon contains a point of a given sublattice.  The Main Theorem stated in 
Section~\ref{sec:mt:mt} provides this formula.

To put it the other way around, the Main Theorem gives an optimal upper bound 
on the number of vertices of a convex lattice polygon free of points of a 
given sublattice.  Clearly, convexity is essential for this bound to exist, 
but we do not impose other requirements on the lattice polygons.

The proof of the Main Theorem can be naturally reduced to estimating the 
number of vertices of integer polygons free of points of the lattice~$\nz$.  
This can be broken up into two major steps.

First, we would like to obtain a feasible description of integer polygons free 
of points of~$\nz = (n \z)^2$.  A crucial property of such polygons is that 
each of them lies in a $\nz$-slab of $\nz$-width~3 
(Proposition~\ref{pr:slab}).  Using this as basis, we classify such polygons 
up to affine transformations preserving the lattice~$\nz$ into six types 
differing by imposed geometric constraints (Definition~\ref{def:types} and 
Theorem~\ref{th:types}).

The second step is to estimate the number of vertices for each type.  This 
requires subtle geometric analysis and can be quite technical in terms of 
computations.  In this paper we develop necessary tools in 
Section~\ref{sec:slopes} and apply them to one particular class of polygons, 
where the estimates can be derived immediately.  More technical cases are the 
subject of \cite{BKb}.

In order to obtain the estimates we break up the boundary of a polygon into 
several broken lines and translate geometrical constraints imposed on the 
polygon into Diaphantine inequalities.  Resulting inequalities relate the 
numbers of edges of the broken lines and the coordinates of their endpoints, 
which are also the parameters of the bounding box of the polygon.  A part of 
this paper is specifically devoted to developing tools for this translation. 
Theorem~\ref{th3-6}, Corollary~\ref{cor3-6}, and Theorem~\ref{th3-8} are the 
most noteworthy results in this direction.

The rest of the paper is organised as follows.

Section~\ref{sec:mt} is devoted to the overview of the results, the Main 
Theorem being stated in Section~\ref{sec:mt:mt} and Section~\ref{sec:synopsis} 
containing a detailed synopsis of the proof.  For convenience, the statement 
of the Main Theorem is split into Sub-Theorems~A,~B, and~C.

In Section~\ref{sec:slopes} we study a class of broken lines we call slopes.  
The first two subsections contain definitions and statements estimating the 
number of edges of a slope.  Then we show how these estimates can be applied 
to polygons and conclude the section by proving Sub-Theorems~C for a 
particular class of polygons.

We tried to keep to a minimum the number of proofs in Sections~\ref{sec:mt} 
and~\ref{sec:slopes}.  Rather, we collected technical proofs in subsequent 
sections.

In Section~\ref{sec:AB} we give fairly simple proofs of Sub-Theorems~A and~B 
that make use of the so-called parity argument and Pick's formula.  An easy 
particular case of Sub-Theorem~C is also proved there.

In Section~\ref{sec:diameter} we study certain properties of the lattice 
diameter of polygons that allow us to prove Proposition~\ref{pr:slab} and 
Theorem~\ref{th:types}.

Section~\ref{sec:pfs} provides the proofs of theorems of 
Section~\ref{sec:slopes}.

\section{Main theorem}
\label{sec:mt}

\subsection{Main theorem}
\label{sec:mt:mt}

Suppose that the system of vectors $\mathbf{a}_1,\ \mathbf{a}_2 \in \rr$ is 
linearly independent; then the set
\[\set{u_1\mathbf{a}_1+u_2\mathbf{a}_2}{u_1,u_2\in\z}\]
is called \emph{a lattice} spanned by $\mathbf{a}_1$, $\mathbf{a}_2$, and  
$\mathbf{a}_1$, $\mathbf{a}_2$ are called the \emph{basis} of the lattice.

\begin{example}
The vectors $\mathbf{e}_1=(1,0),\,\mathbf{e}_2=(0,1)$ span the \emph{integer 
lattice} denoted by $\zz$.  It is the set of points with both integral 
coordinates.  Those are called \emph{integer points}.
\end{example}

A lattice $\Gamma$ is called \emph{a sublattice} of a lattice $\Lambda$ if 
$\Gamma \subset \Lambda$.  If moreover $\Gamma \ne \Lambda$, $\Gamma$ is 
called \emph{a proper sublattice} of $\Lambda$.  In what follows we only 
consider sublattices of the integer lattice.

A lattice $\Lambda \subset \zz$ is spanned by the columns of a matrix $A = 
(a_{ij}) \in GL_2(\z)$ if and only if $\Lambda = A \zz = \set{A\mathbf u}{ 
\mathbf u\in\zz}$. Given $\Lambda$, the matrix $A$ is not uniquely defined.  
However, the numbers
\begin{equation*}
\delta = \gcd(a_{ij}),\ n = |\det A|/\delta
\end{equation*}
are independent of $A$.  They are called \emph{invariant factors} of 
$\Lambda$, and the pair $(\delta, n)$ is \emph{the invariant factor sequence 
of} $\Lambda$ (see e.~g.\ \cite{Nor12}).
\begin{example}
The lattice $n\zz = (n\z) \times (n\z) = \set{(nu_1, nu_2)}{u_1, u_2 \in \z}$, 
where $n$ is a positive integer, has invariant factor sequence $(n, n)$.
\end{example}
\begin{example}
The lattice $\delta \z \times n \zz = \set{(\delta u_1, nu_2)}{u_1, u_2 \in 
\zz}$, where $\delta$ and $n$ are positive integers and $\delta$ divides $n$, 
has invariant factor sequence $(\delta, n)$.
\end{example}

The \emph{convex polygon} is a two-dimensional polytope, i.~e.\ the convex 
hull of a finite set of points that has nonempty interior.  In what follows we 
only consider convex polygons, so we often drop the word `convex'.  We assume 
that the reader is familiar with basic terminology such as vertex and edge, 
see \cite{Gru03, Zie95} for details.  A polygon with $N$ vertices, $N \ge 3$, 
is called an $N$-gon.  The vertices of an \emph{integer polygon} belong to 
$\zz$.  More generally, if all the vertices of a polygon belong to a lattice 
$\Gamma$, it is called a \emph{$\Gamma$-polygon}.  Integer polygons are also 
called \emph{lattice polygons}, but to avoid misunderstanding, we prefer the 
first term, since we consider $\Gamma$-polygons with different lattices 
$\Gamma$.

Given a sublattice $\Lambda$ of $\zz$ with invariant factor sequence $(\delta, 
n)$, define
\[\nu(\Lambda) = \nu(\delta, n) =2n+2\min\{\delta,3\}-3.\]

\begin{maintheorem}
Let $\Lambda$ be a proper sublattice of $\zz$.  Then any convex integer 
polygon with~$\nu(\Lambda)$ vertices contains a point of $\Lambda$.
\end{maintheorem}

It is easily seen that the constant $\nu(\Lambda)$ in the Main Theorem is 
sharp, i.~e.\ if  $\nu(\Lambda) > 3$ for given $\Lambda$, then there exist 
$(\nu(\Lambda)-1)$-gons containing no points of $\Lambda$.  This is very clear 
in case $\Lambda = \delta \z \times n \z$ (see Figure~\ref{fig:existence}).  
The general case follows from the fact that any lattice with invariant factor 
sequence $(\delta, n)$ is the image of $\delta \z \times n \z$ under a linear 
transformation preserving the integer lattice, see Section~\ref{sec:prelim}.

For a synopsis of the proof of the Main Theorem, see 
Section~\ref{sec:synopsis}.

\begin{figure}
\includegraphics{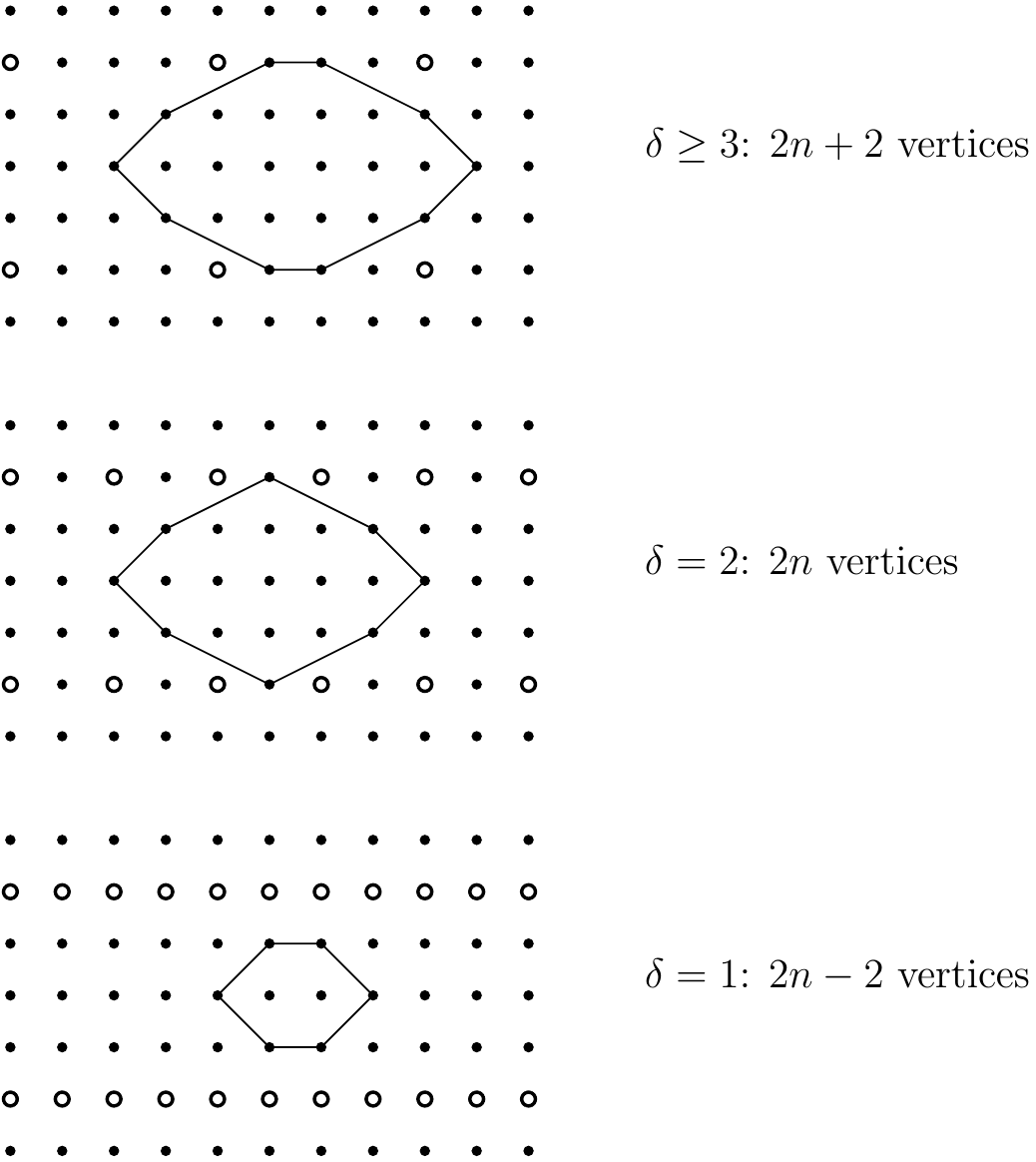}
\caption{If $\delta \ge 3$, it is easy to construct a polygon lying in the 
slab $0 \le x_2 \le n$ and having two vertices on each of the lines $x_2 = j$, 
where $j = 0, \dots, n$, such that the vertices belonging to the lines $x_2 = 
0$ and $x_2 = n$ lie between adjacent points of $\Lambda$.  Clearly, such a 
polygon is free of points of $\Lambda$ and has $2n + 2 = \nu(\delta, n) - 1$ 
vertices.  If $\delta = 1$, the construction is similar, only the polygon 
should have one vertex on each of the lines $x_2 = 0$ and $x_2 = n$ not 
belonging to $\Lambda$.  If $\delta = 0$, it suffices to take any integer 
polygon lying in the slab $1 \le x_2 \le n - 1$ having two vertices on each of 
the lines $x_2  = j$, where $j = 1, \dots, n - 1$.}
\label{fig:existence}
\end{figure}

\subsection{Preliminaries}
\label{sec:prelim}

In this section we list a few familiar properties of lattices.  The proofs can 
be found in \cite{Cas12, GL87, Gru07, BR09}.

We always denote the vectors of the standard basis of~$\rr$ by $\mathbf e_1 = 
(1, 0)$ and $\mathbf e_2 = (0, 1)$ and the standard coordinates in~$\rr$ by 
$x_1$, $x_2$.

Note that any lattice is a subgroup of the additive group of the linear space 
$\rr$ and a free abelian group of rank 2.

A matrix $A \in M_2(\z)$ is called \emph{unimodular}, if $\det A = \pm 1$.  

\begin{proposition}
\label{pr:th1-1}
Let $(\mathbf f_1, \mathbf f_2)$ be a basis of a lattice~$\Lambda$; then the 
vectors $a_{i1} \mathbf f_j + a_{i2} \mathbf f_2$, where $i = 1, 2$, form a 
basis of $\Lambda$ if and only if the matrix $(a_{ij})$ is unimodular.
\end{proposition}

For brevity, we write that $\Lambda$ is a $(\delta, n)$-lattice if it is a 
sublattice of~$\zz$ with invariant factor sequence $(\delta, n)$.  The number 
$\delta n$ is called the \emph{determinant} of~$\Gamma$ and denoted $\det 
\Gamma$.

We use the term `$\Lambda$-point' as a synonym of `point of $\Lambda$'.

A linear transformation of the plane is called a (linear) automorphism of a 
lattice if it maps the lattice onto itself.  It is easily seen that a linear 
transformation is an automorphism of a lattice if and only if it maps some 
(hence, any) basis of the lattice onto another basis.  Consequently, given a 
matrix $A \in M_2(\real)$, the transformation $\mathbf x \mapsto A\mathbf x$ 
is an automorphism of $\zz$ if and only if the matrix~$A$ is unimodular.  We 
call such transformation \emph{unimodular}.  For any positive integer $n$, the 
automorphisms of $\nz$ are exactly unimodular transformations.

Clearly, linear automorphisms of a lattice form a group.

Let $\Lambda$ be a lattice.  A vector~$\mathbf f \in \Lambda$ is called 
\emph{$\Lambda$-primitive}, if any representation $\mathbf f = u \mathbf g$ 
with $g \in \Lambda$ and $u \in \z$ implies $u = \pm 1$.

\begin{proposition}
\label{pr:uni}
Suppose that $\Lambda$ is a lattice and $\mathbf f$ and $\mathbf g$ are 
$\Lambda$-primitive vectors; then there exists an automorphism~$A$ 
of~$\Lambda$ such that $A \mathbf f = \mathbf g$.
\end{proposition}

If $\Lambda$ is a sublattice of $\zz$ and $A$ is a unimodular transformation, 
the image $A\Lambda$ is a lattice with the same invariant factors as 
$\Lambda$.

The following proposition is a fundamental result about unimodular 
transformations.  It is a geometric version of the Smith normal form of 
integral matrices~\cite{Nor12}.

\begin{proposition}
\label{pr:snf}
For any sublattice of $\zz$ with invariant factors $(\delta, n)$ there exists 
a unimodular transformation mapping it onto the lattice $\delta \z \times n 
\z$.
\end{proposition}

An \emph{affine frame} of a lattice~$\Lambda$ is a pair $(\mathbf o; \mathbf 
f_1, \mathbf f_2)$ consisting of a point $\mathbf o \in \Lambda$ and a basis 
$(\mathbf f_1, \mathbf f_2)$ of~$\Lambda$.  An \emph{integer frame} is an 
affine frame of~$\zz$.

An \emph{affine automorphism} of a lattice $\Lambda$ is an affine 
transformation of $\rr$ mapping $\Lambda$ onto itself.  It is not hard to see 
that given $A \in M_2(\real)$ and $\mathbf b \in \rr$, the mapping $\mathbf x 
\mapsto A \mathbf x + \mathbf b$ is an affine automorphism of $\Lambda$ if and 
only if $\mathbf x \mapsto A \mathbf x$ is an automorphism of $\Lambda$ and 
$\mathbf b \in \Lambda$.  In particular, affine automorphisms of $\nz$, where 
$n$ is a positive integer, are exactly the transformations of the form 
$\mathbf x \mapsto A \mathbf x + \mathbf b$, where $A$ is unimodular and 
$\mathbf b \in \nz$.

Of course, if $\m$ is a convex integer $N$-gon and $\varphi$ is an affine 
automorphism of~$\zz$, the image $\varphi(\m)$ is still a convex integer 
$N$-gon.  Obviously, is $\m$ is free of points of a lattice $\Lambda$, then so 
is its image under any affine automorphism of $\Lambda$.

We conclude with a nonstandard definition.

Let $\Lambda$ be a sublattice of~$\zz$ and $(\mathbf f_1, \mathbf f_2)$ be a 
basis of~$\zz$.  Clearly, $\{u \in \z \colon u \mathbf f_1 \in \Lambda\}$ is a 
subgroup of~$\z$.  It is generated by a positive integer, which we call the 
\emph{large $\mathbf f_1$-step} of~$\Lambda$ with respect to $(\mathbf f_1, 
\mathbf f_2)$.  Further, $\{u_1 \in \z \colon \exists\ u_2 \in z,\ u_1 f_1 + 
u_2 f_2 \in \Lambda\}$ is a subgroup of~$\z$, too.  We call its positive 
generator the \emph{small $\mathbf f_1$-step} of~$\Lambda$ with respect to 
$(\mathbf f_1, \mathbf f_2)$.  Alternatively, the small $\mathbf f_1$-step can 
be defined as the largest $s$ such that all the points of~$\Lambda$ lie on the 
lines $\{ks\mathbf f_1 + t\mathbf f_2\}$, $k \in \z$.  Obviously, the small 
step is smaller then the large step.  We can define the large and small 
$\mathbf f_2$-steps of $\Lambda$ with respect to~$(\mathbf f_1, \mathbf f_2)$ 
in the same way.

In what follows we nearly always consider small and large steps of lattices 
with respect to bases made up of the vectors $\pm \mathbf e_1$, $\pm \mathbf 
e_2$, and we usually omit the reference to the basis when there is no 
ambiguity.

\begin{proposition}
\label{pr:steps}
Let $\Lambda$ be a sublattice of~$\zz$ and $(\mathbf f_1, \mathbf f_2)$ be a 
basis of~$\zz$.  Then the product of the small $\mathbf f_1$-step and the 
large $\mathbf f_2$-step of~$\Lambda$ equals $\det \Lambda$.
\end{proposition}

The proof is left to the reader.

In what follows we use standard notations $\floor{\cdot}$ for the floor 
function, $\ceil{\cdot}$ for the ceiling function, $^+$ for the positive part, 
and $|\cdot|$ for the cardinality of a finite set.  As noted above, by 
$[\mathbf a, \mathbf b]$ we denote the segment with the endpoints~$\mathbf a$ 
and~$\mathbf b$.

\subsection{Synopsis of the proof}
\label{sec:synopsis}

It turns out that the Main Theorem can be fairly easily proved for $(1, 2)$- 
and $(2, 2)$-lattices.

In the case of the lattice $\Lambda = \z \times 2 \z$, the Main Theorem 
becomes
\begin{subtheorem-a}
Any convex integer polygon contains a point with an even ordinate.
\end{subtheorem-a}
Sub-Theorem~A implies the Main Theorem for arbitrary $(1, 
2)$-lattices~$\Lambda$, because if $\Lambda$ is such a lattice and $\m$ is an 
integer polygon, we can find a unimodular transformation $A$ such that $A 
\Lambda = \z \times 2\z$ (Proposition~\ref{pr:snf}); Sub-Theorem~A asserts 
that $AP$ contains an $A\Lambda$-point, so $\m$ contains a point of~$\Lambda$.

In the case of the lattice $2\zz$, the Main Theorem becomes
\begin{subtheorem-b}
Any convex integer pentagon contains a point of the lattice~$2\zz$.
\end{subtheorem-b}

This statement was announced in~\cite{Bli00}.

We prove Sub-Theorems~A and~B in Section~\ref{sec:AB}.

In proving the Main Theorem we adopt the strategy of estimating the number of 
vertices (equivalently, of edges) of polygons not containing points of given 
lattices.  We will presently see that we can concentrate on integer polygons 
free of $\nz$-points.  As we can always substitute such a polygon by its image 
under an affine automorphism of $\nz$, our first goal is to find out how 
significantly we can reduce the set of polygons to consider applying such 
automorphisms.  We use the following proposition as our basis.

\begin{proposition}
\label{pr:slab}
Given a convex integer polygon $\m$ free of $\nz$-points, where $n \in \z$, $n 
\ge 2$, there exists an automorphism $\psi$ of $\nz$ such that $\psi(\m)$ 
lies in the slab
$$-n+1\le x_1\le 2n-1.$$
\end{proposition}

The proposition is proved in Section~\ref{sec:diameter}.

\begin{remark}
\label{rem:slab}
The proof of Proposition~\ref{pr:slab} ensures that $\psi$ can be chosen in 
such a way that $\psi(\m)$ contains a segment with $\ell(\m) + 1$ integer 
points lying on a line of the form $x_1 = c$ with $0 \le c \le n$, where 
$\ell(\m)$ is the lattice diameter of~$\m$ (see Section~\ref{sec:diameter}).  
Moreover, $\psi$ can be chosen in such a way that if $\psi(\m)$ has common 
points with the lines $x_1 = 0$ and $x_1 = n$, they lie on the segments 
$[\mathbf 0, (0, n)]$ and $[(n, 0), (n, n)]$, respectively.
\end{remark}
\begin{remark}
If $\Lambda$ is a $(\delta, n)$-lattice, it is not hard to prove that $\nz 
\subset \Lambda$.  Proposition~\ref{pr:slab} immediately implies that the 
number of vertices of a polygon free of $\Lambda$-points cannot be greater 
than $2(3n - 2)$.  Of course, in view of the Main Theorem this fairly simple 
estimate is suboptimal.
\end{remark}

Proposition~\ref{pr:slab} allows for a classification of polygons free of 
points of~$\zz$ into feasible classes.

We say that a line or a segment \emph{splits} a polygon, if it divides the 
polygon into two parts with nonempty interior.

\begin{figure}
\includegraphics{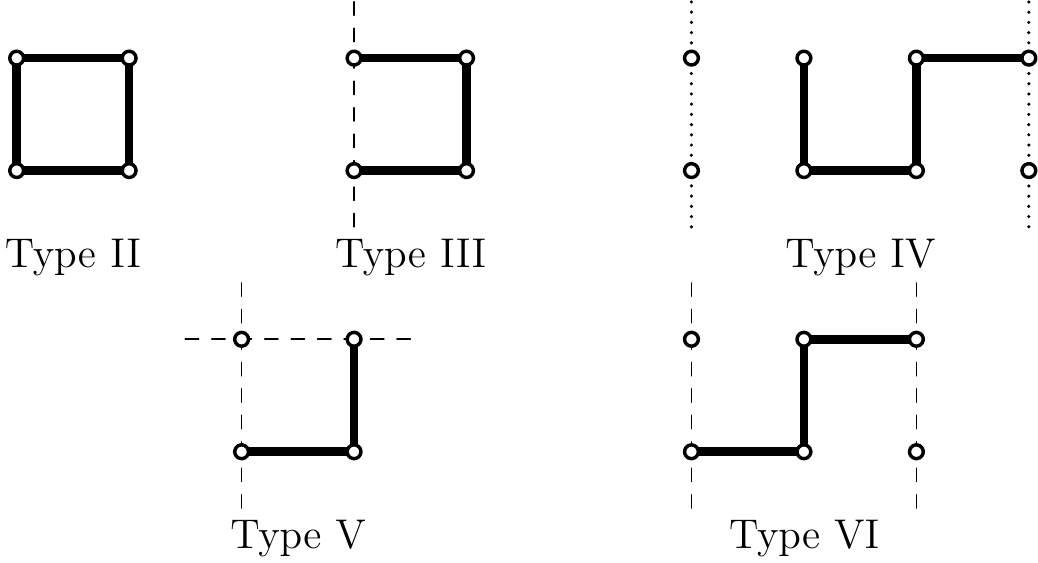}
\caption{Definition~\ref{def:types} introduces the types of polygons in terms 
of intersection with segments and lines.  Here thick segments split polygons 
of the specified type, thin lines do not split them, and dotted lines have no 
common points with them.}
\label{fig:types}
\end{figure}

Let $\m$ be an integer polygon free of points of~$\nz$ and $n$ be an integer, 
$n \ge 2$.

\begin{definition}
\label{def:types}
We say that $\m$ is a
\begin{itemize}
\item
\emph{type~I$_n$ polygon}, if no line of the form $x_1 = jn$ or $x_2 = jn$ 
where $j \in \z$, splits $\m$, or, equivalently, if $\m$ lies in a slab of the 
form $jn \le x_1 \le (j+1)n$ or $jn \le x_2 \le (j+1)n$, where $j \in 
\z$;
\item
\emph{type~II$_n$ polygon}, if each of the segments $[\mathbf{0},(n,0)]$, 
$[(n,0),(n,n)]$, $[(0,n),(n,n)]$, and $[\mathbf{0},(0,n)]$ splits~$\m$;
\item
\emph{type~III$_n$ polygon}, if each of the segments $[\mathbf{0},(n,0)]$, 
$[(n,0),(n,n)]$, and $[(n,n),(0,n)]$ splits~$\m$, and the line $x_1 = 0$ does 
not split $\m$;
\item
\emph{type~IV$_n$ polygon}, if each of the segments $[\mathbf{0},(0,n)]$, 
$[\mathbf 0, (n, 0)]$, $[(n, 0), (n, n)]$, and $[(n, n), (2n, n)]$ splits~$\m$ 
and~$\m$ has no common points with the lines $x_1 = -n$ and $x_n = 2n$;
\item
\emph{type~V$_n$ polygon}, if each of the segments $[\mathbf 0, (-n ,0)]$ and 
$[\mathbf 0, (0, n)]$ splits $\m$ and the lines $x_1 = -n$ and $x_2 = n$ do 
not split~$\m$;
\item
\emph{type~VI$_n$ polygon}, if each of the segments $[\mathbf 0, (-n ,0)]$, 
$[\mathbf 0, (0, n)]$, and $[(0, n), (n, n)]$ splits~$\m$, and the lines $x_1 
= \pm n$ do not split~$\m$.
\end{itemize}
\end{definition}

The polygon types are illustrated on Figure~\ref{fig:types}.

\begin{figure}
\includegraphics{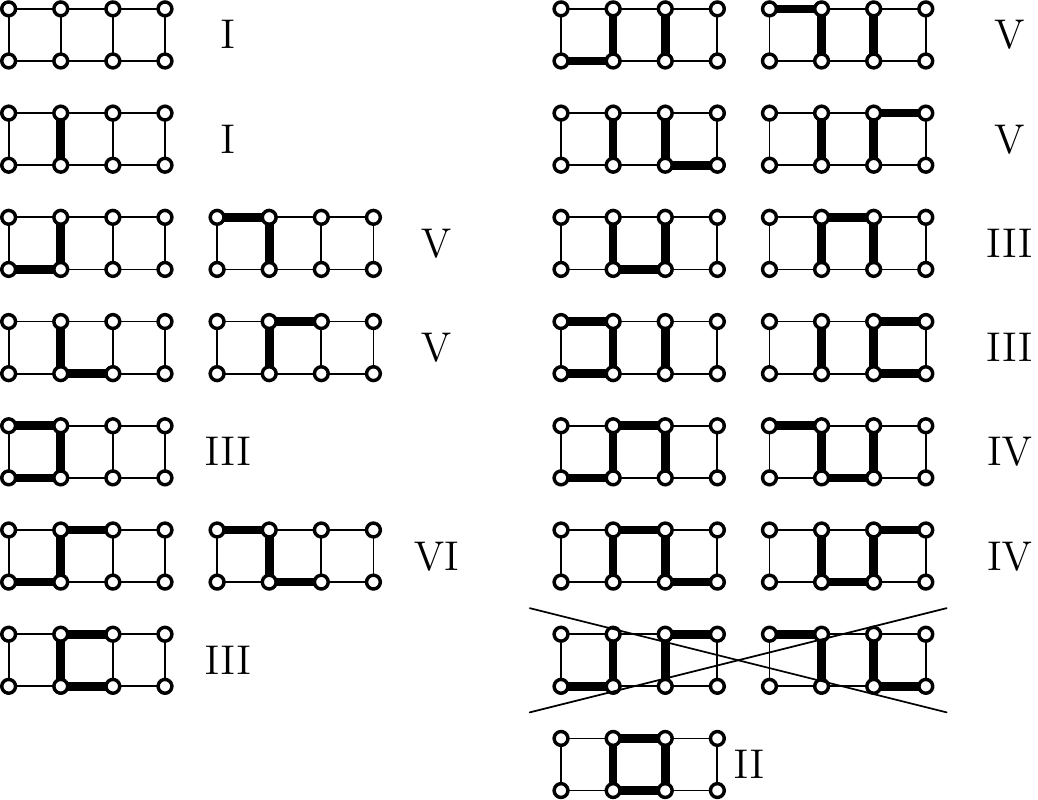}
\caption{The polygon $\m'$ obtained in the proof of Theorem~\ref{th:types} has 
no common points with the lines $x_1 = -n$ and $x_1 = 2n$ and may or may not 
be split by the eight segments.  Two configurations are ruled out by 
Lemma~\ref{lem:imposs}.  For any remaining combination of splitting segments 
(drawn as thick lines), $\m'$ can be easily mapped onto a polygon of specified 
type.}
\label{fig:types-pf}
\end{figure}

\begin{theorem}
\label{th:types}
Suppose that an integer polygon $\m$ is free of points of the lattice $\nz$, 
where $n \in \z$, $n\ge2$; then there exists an affine automorphism $\varphi$ 
of $\nz$ such that $\varphi(\m)$ is a polygon of one of the types 
I$_n$--VI$_n$.
\end{theorem}
\begin{proof}
Let $\psi$ an automorphism as in Proposition~\ref{pr:slab} and 
Remark~\ref{rem:slab} and $\m' = \psi(\m)$.

As $\m'$ is free of~$\nz$-points, it is clear that $\m'$ may be split by at 
most one of the three segments $I_1 = [\mathbf 0, (-n, 0)]$, $I_2 = [\mathbf 
0, (n, 0)]$, and $I_3 = [(n, 0), (2n, 0)]$, and at most one of the three 
segments $J_1 = [(-n, n), (0, n)]$, $J_2 = [(0, n), (n, n)]$, and $J_3 = [(n, 
n), (2n, n)]$.

If the lines $x_1 = 0$ and $x_1 = n$ do not split $\m'$, it is a type~I$_n$ 
polygon.

If exactly one of the lines $x_1 = 0$ and $x_1 = n$, then there is no loss of 
generality in assuming it is the former, because otherwise we can 
replace~$\m'$ by its reflection about the line $x_1 = n/2$, the reflection 
being an automorphism of $\nz$.  Thus, the segment $[\mathbf 0, (0, n)]$ 
splits $\m'$ and the segments $I_3$ and $J_3$ have no common points with 
$\m'$.  Individually examining the possibilities according to which of the 
segments $I_{1,2}$ and $J_{1,2}$ split $\m'$, we see that in each case the 
polygon either is of one of the types I$_n$--VI$_n$ or can be trivially mapped 
onto such a polygon by an automorphism of~$\nz$ (Figure~\ref{fig:types-pf}).

If both lines $x_1 = 0$ and $x_1 = n$ split $\m'$, we likewise consider the 
possibilities according to which of the segments $I_{1,2,3}$ and $J_{1,2,3}$ 
split $\m'$ (Lemma~\ref{lem:imposs} rules out two of them) and draw the same 
conclusion (Figure~\ref{fig:types-pf}).
\end{proof}

Now we recast the Main Theorem for lattices with the largest invariant factor 
greater than~2 as follows:

\begin{subtheorem-c}
Let $\m$ be a convex integer $N$-gon of one of the types~I$_n$--VI$_n$, where 
$n$ is an integer, $n \ge 3$.  Then:
\begin{enumerate}[(i)]
\item
the following inequality holds:
$$
N\le2n+2
;
$$
\item
if the vertices of $\m$ belong to a $(1, n/2)$-lattice, then
$$N\le2n;$$
\item
if the vertices of $\m$ belong to a $(1, n)$-lattice, then
$$N\le2n-2.$$
\end{enumerate}
\end{subtheorem-c}

Let us make sure that having proved Sub-Theorem~C, we in fact establish the 
Main Theorem for $(\delta, n)$-lattices~$\Lambda$ with $n \ge 3$.  Indeed, let 
$\Lambda$ be such a lattice and~$\m$ be an integer $N$-gon with $N \ge 
\nu(\Lambda)$.  Suppose that contrary to our expectations, $\m$ is free of 
points of $\Lambda$.  Let $A$ be a unimodular transformation mapping $\Lambda$ 
onto $\delta \z \times n \z$ and $S$ be the scaling $\diag(n/\delta, 1)$.  The 
superposition $SA$ maps $\Lambda$ onto $\nz$ and $\m$, onto an integer 
polygon~$\m'$ free of points of~$\nz$.  Let $\varphi$ be an affine 
automorphism of~$\nz$ mapping $\m'$ onto a polygon $\m''$ of one of the types 
I$_n$--VI$_n$.  Note that the vertices of~$\m'$ belong to $S\zz = (n/\delta)\z 
\times \z$, so the vertices of~$\m''$ belong to a $(1, n/\delta)$-lattice.  
Now the assumption $N \ge \nu(\delta, n)$ contradicts Sub-Theorem~C applied 
to~$\m''$.

Thus, the Main Theorem is the sum of Sub-Theorems~A, B~and~C.

In the case of type~I polygons the proof of Sub-Theorem~C is a simple 
combinatorial argument, see Section~\ref{sec:AB}.  However, the rest types 
require a fine geometric analysis.  In Section~\ref{sec:slopes} we collect 
necessary tools and apply them to type~II polygons.  The rest types require 
more technical treatment carried out in~\cite{BKb}.

\section{Slopes}
\label{sec:slopes}

\subsection{Slopes}

\begin{figure}
\includegraphics{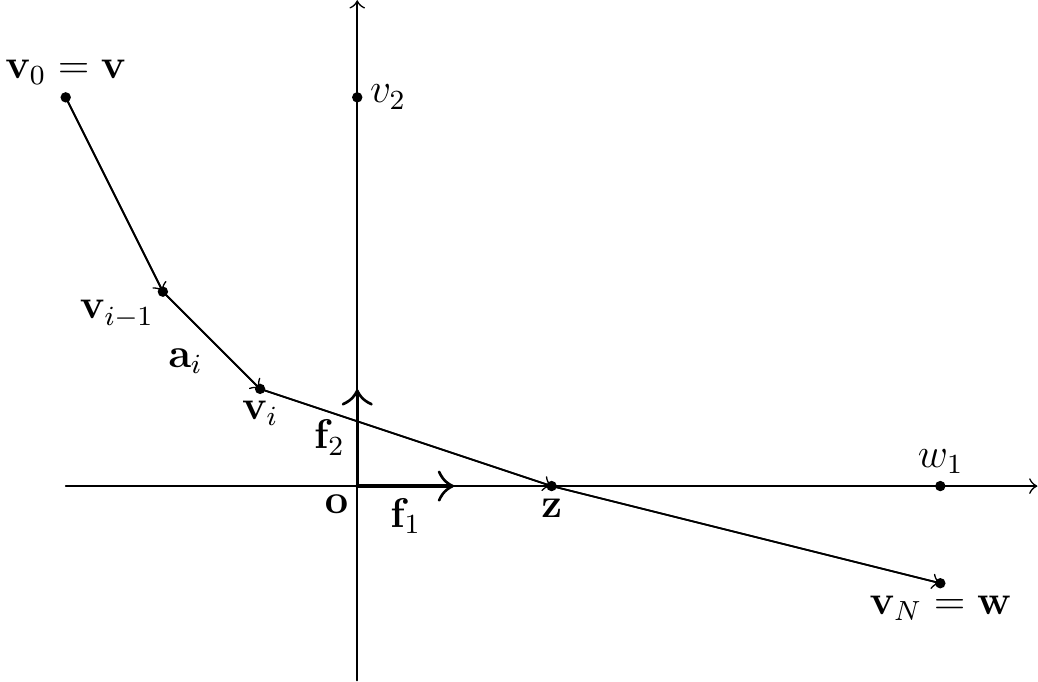}
\caption{The broken line is a slope with respect to the basis $(\mathbf f_1, 
\mathbf f_2)$.  It is convex and the vectors associated with its edges point 
down and to the right.  The frame $(\mathbf o; \mathbf f_1, \mathbf f_2)$ 
splits the slope and forms small angle with it, since there is a supporting 
line passing through the point~$\mathbf z$ and forming an angle $\le \pi/4$ 
with the axis.}
\label{fig:slope}
\end{figure}

Let $(\mathbf{f}_1,\mathbf{f}_2)$ be a basis of $\rr$, and let $\mathbf v_0, 
\mathbf v_1, \dots, \mathbf v_N$ ($N\ge0$) be a finite sequence of points on 
the plane.  If $N\ge1$, set
\begin{equation}\label{eq1-2-3}
\mathbf{v}_i+\mathbf{v}_{i-1} =
\mathbf{a}_i=a_{i1}\mathbf{f}_1+a_{i2}\mathbf{f}_2\qquad
(i=1,\ldots,N)
.
\end{equation}
If
\begin{equation}\label{eq3-3}
a_{i1}>0,\ a_{i2}<0\qquad (i=1,\ldots,N)
\end{equation}
and
\begin{equation}\label{eq4-3}
\begin{vmatrix}
a_{i1}&a_{i+1,1}\\
a_{i2}&a_{i+1,2}
\end{vmatrix}
>0
\qquad (i=1,\ldots,N-1)
,
\end{equation}
we say that the union $Q$ of the segments $[\mathbf v_{0}, \mathbf v_1]$, 
$[\mathbf v_{1}, \mathbf v_2]$, \dots, $[\mathbf v_{N-1}, \mathbf v_N]$ is a 
\emph{slope} with respect to the basis $(\mathbf f_1, \mathbf f_2)$.  These 
segments are called the \emph{edges} of the slope, and the points $\mathbf 
v_0$, $\mathbf v_1$, \dots, $\mathbf v_N$, its \emph{vertices}, $\mathbf v_0$ 
and $\mathbf v_N$ being the \emph{endpoints}.  If $N = 1$, we call the segment 
$[\mathbf v_0, \mathbf v_1]$ a slope if \eqref{eq3-3} holds, and if $N = 0$, 
we still call the one-point set $\{\mathbf v_0\}$ a slope.  If all the 
vertices of $Q$ belong to a lattice $\Gamma$, we call it a 
$\Gamma$-\emph{slope}.  A $\zz$-slope is called \emph{integer}, and it is the 
only kind of slopes we are interested in.

It is not hard to prove that the vertices and edges of a slope are uniquely 
defined, and that the basis induces a unique ordering of vertices.

Figure~\ref{fig:slope} illustrates the concepts of a slope and of an affine 
frame splitting a slope, to be considered below.

\begin{remark}
\label{rem3-4}
If $Q$ is a slope with respect to a basis $(\mathbf f_1, \mathbf f_2)$, then 
it is a slope with respect to the basis $(\mathbf f_2, \mathbf f_1)$, too.
\end{remark}

Although the following statement is simple, it provides handy tools for 
estimating the number of edges of a slope.  We are interesting in comparing 
the doubled number of edges with the `width' of the slope, i.~e.\ its 
projection on the axis spanned by $\mathbf f_1$.  The general point is that 
the edges with projection~1 contribute quadratic growth to the `height' of the 
slope.

\begin{proposition}
\label{pr:slp}
Let $(\mathbf f_1, \mathbf f_2)$ be a basis of~$\zz$ and~$\mathbf v$ 
and~$\mathbf w$ be the endpoints of an integer slope (with respect to 
$(\mathbf f_1, \mathbf f_2)$) having~$N$ edges.  Let
\begin{equation*}
\mathbf w - \mathbf v = b_1 \mathbf f_1 + b_2 \mathbf f_2
.
\end{equation*}
Then there exists an integer $s$ such that
\begin{gather}
2N\le |b_1| + s,
\label{eq:slp1}
\\
|b_2| \ge\frac{s(s+1)}{2},
\label{eq:slp2}
\\
0\le s\le N.
\label{eq:slp3}
\end{gather}
If the vertices of the slope belong to a lattice with small $\mathbf f_1$-step 
greater then~1, one can take $s = 0$, so that
\begin{equation}
\label{eq:slp4}
2N \le |b_1|
.
\end{equation}
If the vertices of the slope belong to a lattice having the basis 
$(\mathbf{f}_1-a\mathbf{f}_2,m\mathbf{f}_2)$, where $1\le
a\le m$, then~\eqref{eq:slp2} can be replaced by
\begin{equation}
\label{eq:slp5}
|b_2| \ge\frac{2a+(s-1)m}{2}s
.
\end{equation}
\end{proposition}
\begin{proof}
Let $\mathbf v_0 = v$, $\mathbf v_1$, \dots, $\mathbf v_N = \mathbf w$ be the 
vertices of the slope and assume that~\eqref{eq1-2-3}--\eqref{eq4-3} hold.  It 
follows from~\eqref{eq3-3} and~\eqref{eq4-3} that $\mathbf a_i \ne \mathbf 
a_j$ for $i \ne j$.  Set $A = \{\mathbf a_i \colon a_{i1} = 1\}$ and $s = 
|A|$.  Observe that~$s$ satisfies~\eqref{eq:slp3} and $s = 0$ if the vertices 
of the slope belong to a lattice with small $\mathbf f_1$-step greater then~1.

Let us prove~\eqref{eq:slp1}.  If $\mathbf a_i \notin A$, we have $a_{i1} \ge 
2$, so
\begin{equation*}
|b_1|
= \sum_{i = 1}^N a_{i1}
= \sum_{\mathbf a \in A} a_{i1}
+ \sum_{\mathbf a \notin A} a_{i1}
\ge
|S| + 2(N - |S|) = 2N - s
,
\end{equation*}
and~\eqref{eq:slp1} follows.

Let us prove~\eqref{eq:slp5} assuming that the slope satisfies correspondent 
hypothesis.  It is easily seen that the vectors belonging to~$A$ are of the 
form $\mathbf f_1 - (a + um) \mathbf f_2$, where $u \in \z$, $u \ge 0$.  Thus,
\begin{multline*}
|b_2|
= \sum_{i = 1}^N (- a_{i1})
\ge \sum_{\mathbf a_i \in A} (- a_{i1})
\\
\ge a + (a + m) + \dots + (a + (s-1)m)
=\frac{2a+(s-1)m}{2}s
,
\end{multline*}
as claimed.

In the case of a generic integer slope, letting $m = a = 1$, we 
recover~\eqref{eq:slp2} from~\eqref{eq:slp5}.
\end{proof}

\subsection{Splitting frames}

Let $(\mathbf o;\mathbf{f}_1,\mathbf{f}_2)$ be an integer frame and $Q$ be a 
slope with respect to $(\mathbf f_1, \mathbf f_2)$.

\begin{definition}
We say that the frame $(\mathbf o;\mathbf{f}_1,\mathbf{f}_2)$ \emph{splits} 
the slope $Q$ if
\begin{enumerate}
\item
one endpoint $\mathbf v = \mathbf o + v_1 \mathbf f_1 + v_2 \mathbf f_2$ of 
$Q$ satisfies
\begin{equation}
\label{eq:sf1}
v_1 < 0, \ v_2 > 0,
\end{equation}
while the other endpoint $\mathbf w = \mathbf o + w_1 \mathbf f_1 + w_2 
\mathbf f_2$ satisfies
\begin{equation}
\label{eq:sf2}
w_1 > 0, \ w_2 < 0;
\end{equation}
\item
there exists a point on $Q$ having both positive coordinates in the frame 
$(\mathbf o;\mathbf{f}_1,\mathbf{f}_2)$.
\end{enumerate}
\end{definition}
\begin{remark}
Obviously, a frame can only split a slope if the slope has at least one edge.
\end{remark}

\begin{remark}
\label{rem:slpq}
If an integer frame $(\mathbf o; \mathbf f_1, \mathbf f_2)$ splits a 
slope~$Q$, it is obvious that~$Q$ has no points in the quadrant~$\{\mathbf o + 
\lambda_1 \mathbf f_1 + \lambda_2 \mathbf f_2 \colon \lambda_1, \lambda_2 \le 
0\}$.
\end{remark}

Suppose that a frame $(\mathbf o; \mathbf f_1, \mathbf f_2)$ splits a 
slope~$Q$ and let~$\mathbf z$ be the point where~$Q$ meets the ray 
$\{\mathbf{o}+\lambda \mathbf{f}_1 \colon \lambda\ge 0\}$.  If there is a 
supporting line for~$Q$ passing through~$\mathbf z$ that forms an angle $\le 
\pi/4$ with the ray, we say that the frame $(\mathbf o; \mathbf f_1, \mathbf 
f_2)$ \emph{forms small angle} with the slope~$Q$.

\begin{proposition}
\label{pr:sa}
Suppose that an integer frame $(\mathbf o; \mathbf f_1, \mathbf f_2)$ splits a 
slope~$Q$; then the frame $(\mathbf o; \mathbf f_2, \mathbf f_1)$ splits it as 
well, and at least one of the frames forms small angle with~$Q$.  If there 
exists a point $\mathbf{y} = \mathbf o + y_1 \mathbf f_1 + y_2 \mathbf f_2 \in 
Q$ such that $y_2 > 0$ and $y_1 + y_2 \le 0$, then $(\mathbf o; \mathbf f_1, 
\mathbf f_2)$ forms small angle with $Q$.
\end{proposition}

The proof is left to the reader.

The following theorems provide much more sophisticated estimates of the number 
of edges of a slope than those of Preposition~\ref{pr:slp}.  This time we are 
comparing the doubled number of edges with the length of the projection of the 
slope on the positive half-axes of the frame, where by the projection on a 
half-axis we mean the intersection of the projection on the axis with the 
half-axis.  It turns out that the doubled number of edges is always less then 
or equal to the total length of the projection.

\begin{theorem}\label{th3-6}
Suppose that an integer frame $(\mathbf o; \mathbf f_1, \mathbf f_2)$ splits 
an integer slope $Q$ having $N$ edges and the endpoints $\mathbf v = \mathbf o 
+ v_1 \mathbf f_1 + v_2 \mathbf f_2$ and $\mathbf w = \mathbf o + w_1 \mathbf 
f_1 + w_2 \mathbf f_2$ satisfying~\eqref{eq:sf1} and~\eqref{eq:sf2}.  Then 
there exist $s \in \z$ and $t \in \z$ such that
\begin{gather}
0\le s\le t,
\label{eq:3-6A}
\\
v_2 - s \ge 0,
\label{eq:3-6B}
\\
- v_1 < ts-\frac{s^2-s}{2}+(v_2 - s)(t+1),
\label{eq:3-6C}
\\
2N\le v_2 + w_1 - t + s.
\label{eq:3-6D}
\end{gather}
Moreover, if $(\mathbf o; \mathbf f_1, \mathbf f_2)$ forms small angle with 
$Q$, we have
\begin{equation}
\label{eq:3-6E}
2N\le
v_2+w_1-t+s-\Ceil{\frac{-w_{2}}{2}} + 1
.
\end{equation}
\end{theorem}

\begin{corollary}
\label{cor3-6}
Under the hypotheses of Theorem~\ref{th3-6},
\begin{equation*}
2N\le v_2+w_1
,
\end{equation*}
and if $(\mathbf o; \mathbf f_1, \mathbf f_2)$ forms small angle with $Q$, 
then
\begin{equation*}
2N\le
v_2+w_1-\Ceil{\frac{-w_{2}}{2}} + 1
.
\end{equation*}
\end{corollary}

\begin{theorem}
\label{th3-8}
Under the hypotheses of Theorem~\ref{th3-6}, if the vertices of~$Q$ belong to 
a proper sublattice of $\zz$, then
\begin{equation*}
2N \le v_2 + w_1 - 1
.
\end{equation*}
\end{theorem}

The proofs of Theorems~\ref{th3-6} and~\ref{th3-8} are rather technical.  We 
give them in Section~\ref{sec:pfs}.

\subsection{The boundary of a convex polygon}
\label{sec:slopes:boundary}

\begin{figure}
\includegraphics{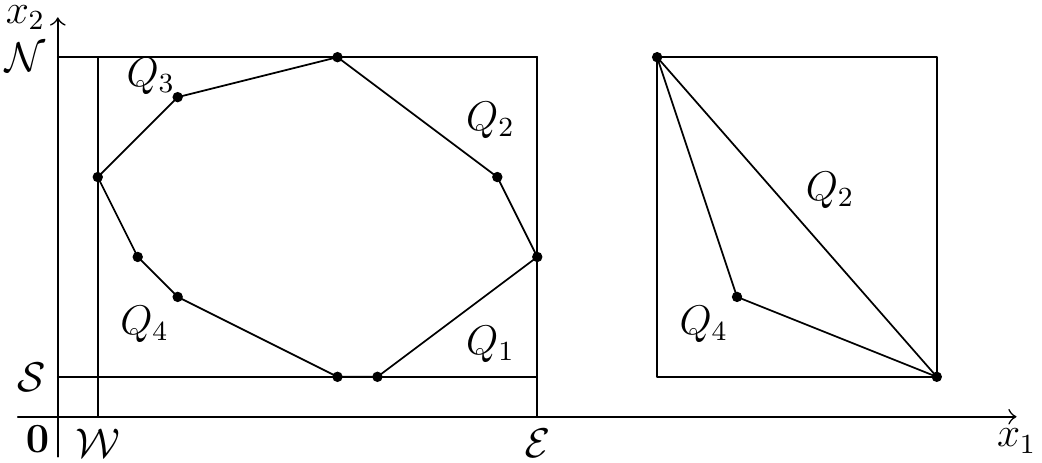}
\caption{The edges of a polygon not belonging to the bounding box form four 
maximal slopes $Q_k$.  These slopes may degenerate into a point, as is the 
case for the triangle on the right, which has only two nontrivial maximal 
slopes.}
\label{fig:maxslopes}
\end{figure}

Let $\m$ be a convex integer polygon.  Define
\begin{gather*}
\begin{array}{l}
\v=\max\{x_2:\;(x_1,x_2)\in\m\}, \\
\vl=\min\{x_1:\;(x_1,\v)\in\m\}, \\
\vp=\max\{x_1:\;(x_1,\v)\in\m\},
\end{array}
\begin{array}{l}
\n=\min\{x_2:\;(x_1,x_2)\in\m\}, \\
\nl=\min\{x_1:\;(x_1,\n)\in\m\}, \\
\np=\max\{x_2:\;(x_2,\n)\in\m\}, \\
\end{array}
\\
\begin{array}{l}
\l=\min\{x_1:\;(x_1,x_2)\in\m\}, \\
\ln=\min\{x_2:\;(\l,x_2)\in\m\}, \\
\lv=\max\{x_2:\;(\l,x_2)\in\m\}, \\
\end{array}
\begin{array}{l}
\p=\max\{x_1:\;(x_1,x_2)\in\m\}, \\
\pn=\min\{x_2:\;(\p,x_2)\in\m\}, \\
\pv=\max\{x_2:\;(\p,x_2)\in\m\}.
\end{array}
\end{gather*}
All these are integers.  Note that $(\nl,\n)$, $(\np,\n)$,
$(\vl,\v)$, $(\vp,\v)$, $(\l,\ln)$, $(\l,\lv)$, $(\p,\pn)$,
and $(\p,\pv)$ are (not necessarily distinct) vertices of $\m$.

There are four slopes naturally associated with a given polygon~$\m$.

Let us enumerate the vertices of~$\m$ starting from $\mathbf v_0 = (\l, \ln)$ 
and going counter-clockwise until we reach $v_{N_4} = (\nl, \n)$.  Clearly, 
the sequence $\mathbf v_0$, \dots, $\mathbf v_{N_4}$ gives rise to a slope 
with respect to the basis $(\mathbf e_1, \mathbf e_2)$.  We denote it 
by~$Q_4$.  Obviously,~$Q_4$ is an inclusion-wise maximal slope with respect to 
$(\mathbf e_1, \mathbf e_2)$ contained in the boundary of~$\m$.  Likewise, we 
define the slope $Q_1$ with respect to $(\mathbf e_2, - \mathbf e_1)$ having 
the endpoints $(\np, \n)$ and $(\p, \pn)$, the slope $Q_2$ with respect to $(- 
\mathbf e_1, - \mathbf e_2)$ having the endpoints $(\p, \pv)$ and $(\vp, \v)$, 
and the slope $Q_3$ with respect to $(- \mathbf e_2, \mathbf e_1)$ having the 
endpoints $(\vl, \v)$ and $(\l, \lv)$.  We call those \emph{maximal slopes} of 
the polygon~$P$ and denote by~$N_k$ the number of edges of~$Q_k$.

\begin{remark}
For each of the mentioned bases, the boundary of the polygon contains 
single-point inclusion-wise maximal slopes apart from correspondent~$Q_k$.  
However, we single~$Q_k$ out by explicitly indicating its endpoints.  For a 
given polygon, some of the maximal slopes~$Q_k$ may have but one vertex.
\end{remark}

Define
\begin{gather*}
M_1 =
\begin{cases}
0, & \text{if } \nl = \np,
\\
1, & \text{otherwise;}
\end{cases}
\ M_2 =
\begin{cases}
0, & \text{if } \pn = \pv,
\\
1, & \text{otherwise;}
\end{cases}
\\
M_3 =
\begin{cases}
0, & \text{if } \vl = \vp,
\\
1, & \text{otherwise;}
\end{cases}
\ M_4 =
\begin{cases}
0, & \text{if } \ln = \lv,
\\
1, & \text{otherwise.}
\end{cases}
\end{gather*}

\begin{proposition}
\label{pr:boundary}
Let~$\m$ be an $N$-gon; then each edge of~$\m$ either lies on a horizontal or 
a vertical line or it is the edge of exactly one of the maximal slopes 
of~$\m$; thus,
\begin{equation*}
N = \sum_{k = 1}^4 N_k + \sum_{k = 1}^4 M_k
.
\end{equation*}
\end{proposition}

The point of Proposition~\ref{pr:boundary} is that if we want to estimate the 
number of edges of a polygon, we can do so by considering its maximal slopes 
and applying the techniques presented above.  The following statement is a 
helpful sufficient condition for a frame to split a maximal slope.

\begin{proposition}
\label{pr:th5-3}
Let $\m$ be a convex integer polygon and $(\mathbf o; \mathbf f_1, \mathbf 
f_2)$ be an integer frame such that $\mathbf f_1, \mathbf f_2 \in \{\pm 
\mathbf e_1, \pm \mathbf e_2\}$.  Suppose that $\mathbf o$ does not belong 
to~$\m$ and the rays $\{\mathbf{c} +\lambda
\mathbf{f}_j \colon \lambda\ge 0\}$ ($j=1,2$) split $\m$; then $(\mathbf o; 
\mathbf f_1, \mathbf f_2)$ splits $Q_k$, where
$$k=\left\{\begin{array}{lllll}
1,&\text{if}&(\mathbf f_1, \mathbf f_2)=(-\mathbf{e}_1,\mathbf{e}_2)&\text{or}&(\mathbf f_1, \mathbf f_2)=(\mathbf{e}_2,-\mathbf{e}_1),\\
2,&\text{if}&(\mathbf f_1, \mathbf f_2)=(-\mathbf{e}_2,-\mathbf{e}_1)&\text{or}&(\mathbf f_1, \mathbf f_2)=(-\mathbf{e}_1,-\mathbf{e}_2),\\
3,&\text{if}&(\mathbf f_1, \mathbf f_2)=(\mathbf{e}_1,-\mathbf{e}_2)&\text{or}&(\mathbf f_1, \mathbf f_2)=(-\mathbf{e}_2,\mathbf{e}_1),\\
4,&\text{if}&(\mathbf f_1, \mathbf f_2)=(\mathbf{e}_2,\mathbf{e}_1)&\text{or}&(\mathbf f_1, \mathbf f_2)=(\mathbf{e}_1,\mathbf{e}_2).\\
\end{array}\right.$$
\end{proposition}

The following simple statement also proves useful.

\begin{proposition}
\label{pr:th5-1}
Let $\m$ be a $\Gamma$-polygon, $S_1$ be the large $\mathbf{e}_1$-step of 
$\Gamma$, and $S_2$ be the large $\mathbf{e}_2$-step of $\Gamma$.  Then
\begin{gather*}
\np-\nl\ge S_1M_1,\ \pv-\pn\ge S_2M_2,\\
\vp-\vl\ge S_1M_3,\ \lv-\ln\ge S_2M_4.
\end{gather*}
\end{proposition}

The proofs of Propositions~\ref{pr:boundary}, \ref{pr:th5-3}, 
and~\ref{pr:th5-1} are left to the reader.

\subsection{Application to type~II polygons}

In this section we present a demonstration of the tools developed above by 
proving Sub-Theorem~C for type~II polygons.

\begin{lemma}\label{lem7-1}
Suppose that $n \ge 3$ is an integer and $\m$ is a type~II$_n$ polygon; then
\begin{enumerate}[(i)]
\item $((n,0);-\mathbf{e}_1,\mathbf{e}_2)$ splits $Q_1$;
\item $((n,n);-\mathbf{e}_1,-\mathbf{e}_2)$ splits $Q_2$;
\item $((0,n);\mathbf{e}_1,-\mathbf{e}_2)$ splits $Q_3$;
\item $(\mathbf{0};\mathbf{e}_1,\mathbf{e}_2)$ splits $Q_4$;
\item if all the vertices of $\m$ belong to a $(1,n)$-lattice $\Gamma$, then 
the large $\mathbf{e}_1$-step and large $\mathbf{e}_2$-step of $\Gamma$ are 
greater then or equal to $2$.
\end{enumerate}
\end{lemma}
\begin{proof}
Statements (i)--(iv) immediately follow from the definition of a type~II$_n$ 
polygon and Proposition~\ref{pr:th5-3}.  To prove (v), note that the $\det 
\Gamma = n$, so by Proposition~\ref{pr:steps} it suffices to show that the 
small $\mathbf e_1$- and $\mathbf e_2$-steps of $\Gamma$ are less then $n$.  
Obviously, the vertex $(\nl, \n)$ of~$\m$ lies in the slab $0 < x_1 < n$ (this 
follows, for example, from~(i) and~(ii)) and belongs to~$\Gamma$, so the small 
$\mathbf e_1$-step of~$\Gamma$ is indeed less than~$n$.  Likewise, the vertex 
$(\l, \ln)$ lies in the slab $0 < x_2 < n$, so the small $\mathbf e_2$-step 
of~$\Gamma$ is less than~$n$ as well.
\end{proof}

\begin{proof}[Proof of Sub-Theorem~C for type~II$_n$ polygons]
Assume that $\m$ is a type~II$_n$ $N$-gon whose vertices belong to~$\Gamma$, 
where either $\Gamma=\zz$, or $\Gamma$ is a $(1,n/2)$-lattice (which is only 
possible if $n$ is even), or a $(1,n)$-lattice.  Define $b$ as follows:
\begin{equation*}
b=
\begin{cases}
0,&\text{if } \Gamma=\zz,\\
1, &\text{if $\Gamma$ is a $(1, n/2)$-lattice}, \\
2, &\text{if $\Gamma$ is a $(1, n)$-lattice}.
\end{cases}
\end{equation*}
It suffices to show that
\begin{equation}\label{eq*-lem7-2}
N\le2n+2-2b.
\end{equation}

We begin by translating the geometrical constraints on $P$ into inequalities.

Evoking Corollary~\ref{cor3-6} and Theorem~\ref{th3-8} for the maximal slopes 
of~$\m$ and correspondent frames indicated in Lemma~\ref{lem7-1}, we obtain:
\begin{gather*}
2N_1\le -\np+\pn+n+\frac{b^2-3b}{2}, \\
2N_2\le -\vp-\pv+2n+\frac{b^2-3b}{2}, \\
2N_3\le \vl-\lv+n+\frac{b^2-3b}{2}, \\
2N_4\le \nl+\ln+\frac{b^2-3b}{2},
\end{gather*}
where the term $(b^2-3b)/2$ is chosen in such a way that it vanishes at $b = 
0$ and equals $-1$ at $b=1$ and $b=2$.  Further, by Proposition~\ref{pr:th5-1} 
we obtain
\begin{gather*}
\np-\nl\ge\frac{b^2-b+2}{2}M_1, \\
\pv-\pn\ge\frac{b^2-b+2}{2}M_2, \\
\vp-\vl\ge\frac{b^2-b+2}{2}M_3, \\
\lv-\ln\ge\frac{b^2-b+2}{2}M_4,
\end{gather*}
since if $b = 0$ or $b = 1$, the large $\mathbf e_1$- and $\mathbf e_2$-steps 
of~$\Gamma$ are at least~1, and if $b = 2$, by Lemma~\ref{lem7-1} we have that 
those steps are at least~2.

Using the above inequalities, we obtain:
\begin{multline*}
2N=\sum_{k=1}^42N_k+\sum_{k=1}^42M_k\le\left(-\np+\pn+n+\frac{b^2-3b}{2}\right)
\\
+\left(-\vp-\pv+2n+\frac{b^2-3b}{2}\right)+\left(\vl-\lv+n+\frac{b^2-3b}{2}\right)
\\
+\left(\nl+\ln+\frac{b^2-3b}{2}\right)+2M_1+2M_2+2M_3+2M_4=4n+2b^2-6b
\\
+\left(\frac{b^2-b+2}{2}M_1-(\np-\nl)\right)+\left(\frac{b^2-b+2}{2}M_2-(\pv-\pn)\right)
\\
+\left(\frac{b^2-b+2}{2}M_3-(\vp-\vl)\right)+\left(\frac{b^2-b+2}{2}M_4-(\lv-\ln)\right)
\\
+\frac{-b^2+b+2}{2}(M_1+M_2+M_3+M_4)
\\
\le 4n+2b^2-6b
+\frac{-b^2+b+2}{2}(M_1+M_2+M_3+M_4).
\end{multline*}
Observe that $(-b^2+b+2)/2\ge0$ for $b=0,1,2$, so we can proceed as follows:
$$2N\le 4n+2b^2-6b+\frac{-b^2+b+2}{2}\cdot 4=4n+4-4b,$$
which yields \eqref{eq*-lem7-2}.
\end{proof}

\section{Proof of Sub-Theorems A and B and of Sub-Theorem C for type I 
polygons}
\label{sec:AB}

\begin{proof}[Proof of Sub-Theorem A]
Assume that, contrary to our claim, there exists an integer polygon~$\m$ 
containing no integer points with even ordinates.  Let $T \subset \m$ be an 
integer triangle having no integer points apart from its vertices 
$\mathbf{a}=(a_1,a_2)$, $\mathbf{b}=(b_1,b_2)$, and $\mathbf{c}=(c_1,c_2)$.  
Clearly, the numbers $a_2$, $b_2$, and $c_2$ are odd.  Consequently, the area 
of $T$ is an integer number, as up to sign it equals
$$
\frac12
\begin{vmatrix}
b_1-a_1&b_2-a_2\\c_1-a_1&c_2-a_2
\end{vmatrix}
=
\begin{vmatrix}
b_1-a_1&(b_2-a_2)/2\\c_1-a_1&(c_2-a_2)/2
\end{vmatrix}
.
$$
However, by Pick's theorem the area of~$T$ equals $i + b/2 - 1 = 1/2$, where 
$i = 0$ is the number of integer points belonging to the interior of~$T$ and 
$b = 3$ is the number of integer points on the boundary; a contradiction.
\end{proof}

\begin{figure}
\includegraphics{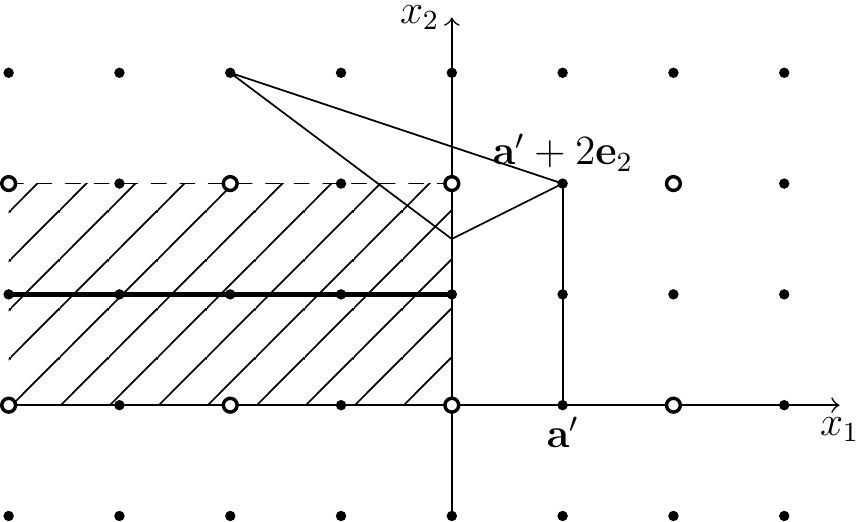}
\caption{The convex hull of the point $\mathbf a' = (1, 0)$, any point 
satisfying $x_1 < 0$ and $x_2 \le 0$, and any point of the segment $[(0, 0), 
(0, 2)]$ contains the point $(0, 0)$.  Likewise, the convex hull of the point 
$\mathbf a' + 2\mathbf e_2 = (1, 2)$, any point satisfying $x_1 < 0$ and $x_2 
\ge 2$, and any point of the segment $[(0, 0), (0, 2)]$ contains the point 
$(0, 2)$.  Thus, if a polygon is free of $2\zz$-points, contains $(1, 0)$ and 
$(1, 2)$ and has common points with $[\mathbf 0, (0, 2)]$, it cannot contain 
any point from the left half-plane not belonging to the hatched slab.  
Consequently, all the integer points of such a polygon belonging to the left 
half-plane lie on the line $x_2  = 1$.}
\label{fig:2Z2}
\end{figure}

\begin{proof}[Proof of Sub-Theorem B]
Conversely, suppose that $\m$ is an integer pentagon free of $2\zz$-points.  
We use the so-called parity argument based on the fact that the index of 
$2\zz$ in $\zz$ is~4.  This implies that the pentagon has two (distinct) 
vertices $\mathbf a \equiv \mathbf b \pmod{2\zz}$.  Consequently, 
$\mathbf{b}-\mathbf{a}=u\mathbf{f}$, where $\mathbf f$ is a $\zz$-primitive 
vector, and $u\ge 2$ is an integer.  Therefore, the segment $[\mathbf a, 
\mathbf b]$ contains at least three integer points.  By 
Proposition~\ref{pr:uni}, there exists a unimodular transformation~$A$ such 
that $A \mathbf f = \mathbf e_2$, then the segment $A[\mathbf a, \mathbf b]$ 
lies on a line $x_1 = c$, where $c \in \z$.  Observe that $c$ is odd, since 
otherwise every second integer point of the line would belong to $2\zz$, and 
thus the segment would contain a point of this lattice.

Let $T$ be the translation by the vector $(1 - c, 0) \in 2\zz$, then the 
points $\mathbf a' = TA \mathbf a$ and $\mathbf b' = TA \mathbf b$ lie on the 
line $x_1 = 1$.  They are vertices of the pentagon $\m' = TA \m$, which is 
still free of $2\zz$-points.  Clearly, one of the half-planes $x_1 < 1$ and 
$x_1 > 1$ (for definitiveness, the former) contains at least two vertices 
of~$\m'$.  These are integer points, so actually they lie in the half-plane 
$x_1 \le 0$.  Consequently, $\m'$ has common points with the line $x_1 = 0$, 
all lying between a pair of adjacent $2\zz$-points, say, on the segment  
$I=[(0,2m),(0,2m+2)]$, where $m$ is an integer.  We can certainly assume that 
$m = 0$, for if not, we replace $\m'$ by $\m' - (0, 2m)$.  Moreover, without 
loss of generality, $\mathbf a' = (1, 0)$, because if $\mathbf a' = (1, 
a'_2)$, we can replace~$\m'$ by $B\m'$, where
\begin{equation*}
B =
\begin{pmatrix}
1 & 0 \\
-a'_2 & 1\\
\end{pmatrix}
\end{equation*}
(note that $B$ preserves the first component of the vectors, so applying it we 
do not break previous assumptions).  To sum up, there is no loss of generality 
in assuming that the polygon $\m'$ contains the points $\mathbf a' = (1, 0)$, 
and $\mathbf a' + 2\mathbf e_2 = (1, 2)$, has common points with the segment 
$[\mathbf 0, (0, 2)]$ and has two vertices satisfying $x_1 \le 0$.  However, 
this is impossible, since the vertices of $\m'$ belonging to the said 
half-plane must lie on the line $x_2 = 1$ (see Figure~\ref{fig:2Z2}).
\end{proof}

\begin{proof}[Proof of Sub-Theorem~C for type~I$_n$ polygons]
For simplicity, assume that~$\m$ lies in the slab $0 \le x_1 \le n$.

All the integer points of the slab lie on $n + 1$ lines, so $N \le 2(n+1)$.

If the vertices of $\m$ belong to a lattice having small $\mathbf e_1$-step $s 
\ge 2$, we have
\begin{equation*}
N \le 2\left(\frac ns + 1\right)
\le 2\left(\Floor{\frac n2} + 1\right) \le 2n - 2
.
\end{equation*}

Now assume that $\m$ is a $\Gamma$-polygon, where $\Gamma$ is a lattice with 
small $\mathbf e_1$-step~1.

If $\Gamma$ is a $(1, n)$-lattice, by Proposition~\ref{pr:steps} the large 
$\mathbf e_2$-step of~$\Gamma$ is~$n$.  Consequently, all the points 
of~$\Gamma$ lying on the lines $x_1 = 0$ and $x_1 = n$ belong to $\nz$.  Thus, 
all the vertices of $\m$ lie on the $n-1$ lines
\begin{equation}
\label{eq:type1-proof}
x_1 = j \quad (j = 1, \dots, n - 1)
,
\end{equation}
whence $N \le 2(n-1)$.

If $\Gamma$ is a $(1, n/2)$-lattice, then the large $\mathbf e_2$-step 
of~$\Gamma$ is~$n/2$.  This implies that on the lines $x_1 = 0$ and $x_1 = n$ 
there is a single point of~$\Gamma$ between adjacent points of $\nz$.  Thus, 
$\m$ has at most 1 vertex on each of these lines and at most 2 vertices on 
each of the lines~\eqref{eq:type1-proof}, totalling at most $2(n-1) + 2 = 2n$ 
vertices.
\end{proof}

\section{Lattice diameter and lattice width}
\label{sec:diameter}

In this section we study properties of integer polygons free of $\nz$-points 
related to their lattice diameter and lattice width.  Our aim is to prove 
Proposition~\ref{pr:slab} and supply what is necesary for the proof 
of~Theorem~\ref{th:types}.

Following~\cite{BF01}, let us introduce

\begin{definition}
The \emph{lattice diameter} of an integer polygon $\m$ is
\begin{equation}
\ell(\m) = \max\{|P \cap \zz \cap L| - 1\}
,
\end{equation}
where the maximum is taken over all the straight lines~$L$ in the plane.
\end{definition}

Clearly, if we consider all strings of integer points in line , $\ell(\m) + 1$ 
is the greatest possible length of a string of integer points in a line 
contained in $\m$.

Affine automorphisms of~$\zz$ preserve the lattice diameter of polygons.

The following lemma provides a simple estimate of the lattice width of a an 
integer polygon~$\m$ in terms of its lattice diameter.  For the sake of 
completeness, we include the proof, even though it is implied by the reasoning 
used in the proof of Theorem~2 of~\cite{BF01}.

\begin{lemma}
\label{lem:th2-1}
Suppose that the convex integer polygon $\m$ contains the points~$\mathbf 0$ 
and $(0, \ell(\m))$.  Then $\m$ lies in the slab
\begin{equation*}
|x_1| \le \ell(\m) + 2
.
\end{equation*}
Moreover, no integer point lying on the lines $x_1 = \pm (\ell(\m) + 1)$ 
belongs to~$\m$.
\end{lemma}
\begin{proof}
Set $\ell = \ell(\m)$ and $\mathbf b = (0, \ell)$.  By convexity, $\m$ 
contains the segment $[\mathbf 0, \mathbf b]$ having the integer points 
$\mathbf 0$, $(0, 1)$, \dots, $(0, \ell)$.

Let us show that no integer point of the lines $x_1= \pm (\ell+1)$ belongs 
to~$\m$. Consider the point $\mathbf{z}=(\eps (\ell + 1),z)$, where $\eps=\pm 
1$ and~$z$ is an arbitrary integer.  Let
$$z=(\ell + 1)q+r,\qquad q,r\in\z,\;0\le r\le \ell.$$
If the point~$\mathbf{z}$ belonged to~$\m$, by convexity the polygon would 
contain the segment $[r\mathbf{e}_2,\mathbf{z}]$ having $\ell + 2$ integer 
points $r\mathbf{e}_2+j(\eps \mathbf{e}_1+q\mathbf{e}_2)$, where  $j=0,\ldots, 
\ell + 1$, which contradicts the definition of the lattice diameter.  
Consequently, $\mathbf z\notin \m$, as claimed.

Let us show that $\m$ is contained in the half-plane $x_1 \le \ell + 2$.

Let $\mathbf v = (v_1, v_2) \in \zz$ be the rightmost vertex of $\m$.  There 
is no loss of generality in assuming that $0 \le v_1 < v_2$, for otherwise we 
could replace~$\m$ by its image under a suitable unimodular transformation 
having a lower-triangular matrix (the line $x_1 = 0$ is invariant under such 
transformations, so the hypotheses of the theorem persist).

We must prove that $v_1 \le \ell + 2$.  If $v_1 \le \ell + 1$, there is 
nothing to prove, so assume that $v_1 \ge \ell + 2$.  By convexity, $\m$ 
contains the triangle $T$ having the vertices $\mathbf 0$, $\mathbf b$, and 
$\mathbf v$.  The line $x_1 = \ell + 1$ intersects~$T$ by a segment~$I$ of 
length
\begin{equation*}
d =
\frac{
\ell(v_1 - \ell - 1)
}{
v_1
}
.
\end{equation*}
First, suppose that $\ell \ge 2$.  If $d\ge 1$, $I$ necessarily contains at 
least one integer point of the line $x_1= \ell + 1$, which is impossible by 
the above.  Consequently, we have
\begin{equation*}
\frac{
\ell(v_1 - \ell - 1)
}{
v_1
}
< 1
.
\end{equation*}
The numerator and the denominator are positive integers, so we get
\begin{equation*}
\ell (v_1 - \ell - 1) \le v_1 - 1
,
\end{equation*}
whence
\begin{equation}
\label{eq:maxseg1}
v_1\le \ell + 2 + \frac{1}{\ell-1}
.
\end{equation}
As $v_1$ is an integer, this yields the desired inclusion provided that $\ell 
\ge 3$.

If $\ell = 2$, inequality~\eqref{eq:maxseg1} becomes $v_1 \le 5$.  However, 
setting $v_1 = 5$ and checking possible values $v_2 = 0, 1, \dots, 4$, we see 
that $T$ invariably contains an integer point lying on the line $x_1 = \ell + 
1$ (Figure~\ref{fig:l=2}).  This is impossible, so actually $v_1 \le 4$, as 
claimed.

For the case $\ell = 1$, see Figure~\ref{fig:l=2}.

To prove that $\m$ is contained in the half-plane $x_1 \ge -\ell - 2$, it 
suffices to reflect $\m$ about the line $x_1 = 0$ and apply the established 
part of the theorem.
\end{proof}

\begin{figure}
\includegraphics{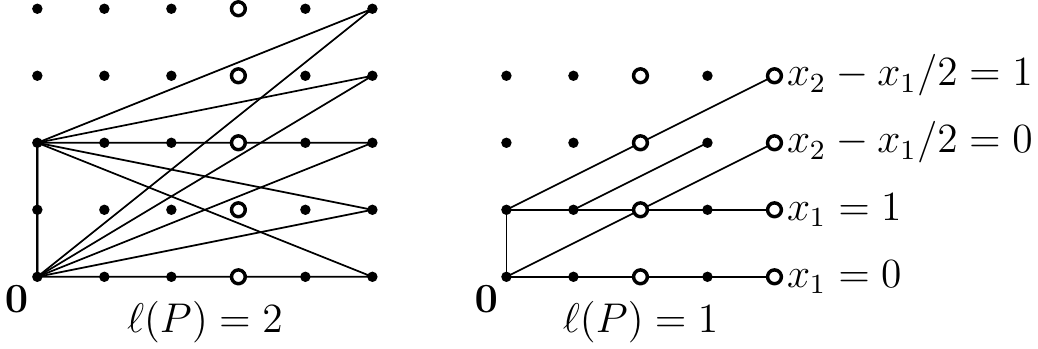}
\caption{Proof of Lemma~\ref{lem:th2-1}.  Assuming that $\ell(\m) = 2$ and 
$v_1 = 5$, we consider all possible cases for $\mathbf v$ and see that the 
triangle~$T$ invariably contains an integer point lying on the line $x_1 = 3$, 
which is impossible.  Now assume that $\ell(\m) = 1$.  As the common points 
of~$T$ and the line $x_1 = 2$ must lie between adjacent integer points, we see 
that $\mathbf v$ must lie either in the slab $0 < x_2 < 1$ or in the slab $0 < 
x_2 - x_1/2 < 1$.  The former case is clearly impossible.  In the latter case, 
because $T$ cannot contain a segment with more than two integer points, we see 
that $v_1 \le 3$, as claimed.
}
\label{fig:l=2}
\end{figure}

\begin{lemma}
\label{lem2-1}
Suppose that an integer polygon $\m$ is free of $\nz$-points, where $n \in 
\z$, $n \ge 2$, and that $\m$ has an integer segment with $\ell(P) + 1$ 
integer points lying on the line $x_1=c$, where
\begin{equation}
\label{eq1-lem2-1}
0\le c\le n
;
\end{equation}
then $\m$ is contained in the slab
$$-n+1\le x_1\le 2n-1.$$
\end{lemma}

\begin{proof}
Set $\ell = \ell(\m)$ and let $\mathbf a = (c, a)$ and $\mathbf b=(c, a + 
\ell)$ be the endpoints of the segment mentioned in the hypothesis of the 
lemma.

Applying Lemma~\ref{lem:th2-1} to the polygon~$\m - \mathbf a$, we see 
that~$\m$ lies in the slab $|x_1 - c| \le \ell + 2$ and has no common integer 
points with the lines $x_1 = c \pm (\ell + 1)$.

Let us show that $\m$ lies in the half-plane $x_1 \le 2n-1$.  Assume the 
converse.  Then the rightmost vertex $\mathbf v = (v_1, v_2)$ of $\m$ 
satisfies
\begin{equation}
\label{eq2-lem2-1}
2n \le v_1 \le c + \ell + 2, \qquad v_1 \ne c + \ell + 1
.
\end{equation}
Consequently, the lines $x_1 = n$ and $x_1 = 2n$ intersect $\m$.  Clearly, the 
intersections must lie between pairs of adjacent points of $\nz$ belonging to 
respective lines, i.~e.\ on some segments $I_1=[(n,u_1n),(n,(u_1+1)n)]$ and 
$I_2=[(2n,u_2n),(2n,(u_2+1)n)]$, where $u_1, u_2\in\z$.  There is no loss of 
generality in assuming that $I_1 = [(n, 0), (n, n)]$ and $I_2 = [(2n, 0), (2n, 
n)]$, since otherwise we could replace $\m$ by its image under the affine 
automorphism $\varphi$ of $\nz$ given by
\begin{equation*}
\varphi(x_1,x_2)=(x_1,(u_1-u_2)x_1+x_2+(u_2-2u_1)n),
\end{equation*}
which does not affect the first coordinate and maps $I_1$ onto $[(n, 0), 
(n,n)]$ and $I_2$ onto $[(2n, 0), (2n, n)]$.

Let us show the inequality
\begin{equation}
\label{eq5-lem2-1}
v_1\le 2n+1.
\end{equation}

The intersection of the line $x_1 = n$ with the triangle with the 
vertices~$\mathbf a$, $\mathbf b$,~and~$\mathbf v$ is a segment $J \subset 
I_1$.  We have:
\begin{equation}
\label{eq4-lem2-1}
\frac{\ell(v_1-n)}{v_1-c}<n
,
\end{equation}
where the left-hand side is the length of $J$.

Assuming that~\eqref{eq5-lem2-1} is not valid, we have $v_1 \ge 2n + 2$.  
By~\eqref{eq2-lem2-1}, we have $v_1 - c \le \ell + 2$ and $\ell \ge v_1 - c - 
2 \ge n$;  thus,
\begin{multline*}
\frac{\ell(v_1-n)}{v_1-c}
\ge
\frac{\ell(n+2)}{\ell + 2}
\\
=
\left(
1-\frac{2}{\ell + 2}
\right)
(n + 2)
\ge
\left(
1-\frac{2}{n+2}
\right)
(n + 2)
=n
,
\end{multline*}
contrary to~\eqref{eq4-lem2-1}, and~\eqref{eq5-lem2-1} is proved.

In view of~\eqref{eq2-lem2-1} and~\eqref{eq5-lem2-1} we have only two possible 
values for~$v_1$: $v_1=2n$ and $v_1=2n+1$.

Further, \eqref{eq4-lem2-1} gives
$$ \ell <\frac{n}{v_1-n}(v_1-c)\le v_1-c,$$
whence
$$\ell \le v_1-c - 1.$$
Comparing this with \eqref{eq2-lem2-1}, we see that necessarily
\begin{equation}
\label{eq6-lem2-1}
\ell =v_1-c-2.
\end{equation}

\begin{figure}
\includegraphics{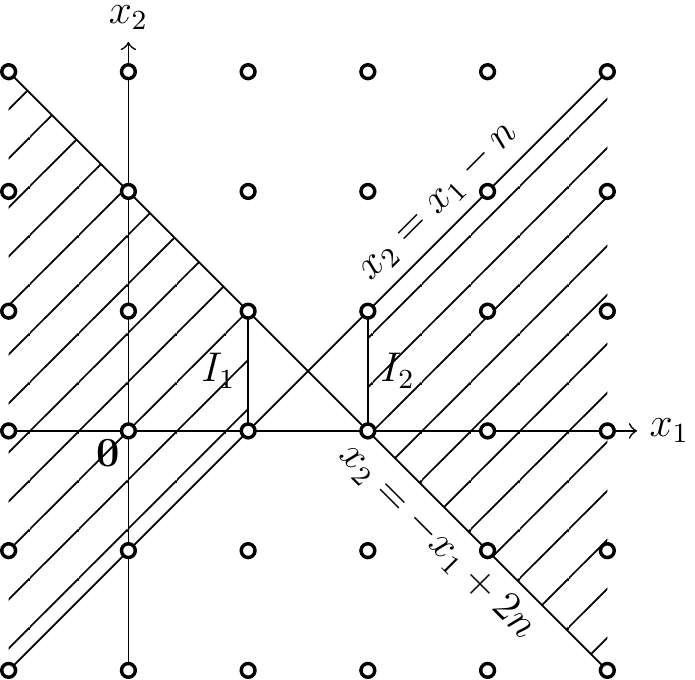}
\caption{Proof of Lemma~\ref{lem2-1}.  The lines passing through $\mathbf a$ 
and $\mathbf v$ and through $\mathbf b$ and $\mathbf v$ belong to the set of 
lines joining interior points of $I_1$ with interior points of $I_2$, so the 
$\mathbf a$ and~$\mathbf b$ belong to the hatched region on the left and 
$\mathbf v$ belongs to the one on the right. }
\label{fig:lem21}
\end{figure}

Let us estimate~$v_2$ and~$a$.  It is easily seen (see Figure~\ref{fig:lem21}) 
that the coordinates of $\mathbf v$, $\mathbf a$, and $\mathbf b$ satisfy
\begin{gather}
-v_1+2n<v_2<v_1-n,
\label{eq7-lem2-1}
\\
a > c-n,
\label{eq8-lem2-1}
\\
a+\ell<-c+2n.
\label{eq9-lem2-1}
\end{gather}
From \eqref{eq7-lem2-1} we get
\begin{equation}\label{eq10-lem2-1}
\left.\begin{array}{lll}1\le v_2\le
n-1&\text{if}&v_1=2n,\\0\le v_2\le
n&\text{if}&v_1=2n+1,\end{array}\right\}
\end{equation}
whereas from~\eqref{eq8-lem2-1} and~\eqref{eq9-lem2-1}  we obtain
\begin{gather}
a+ \ell \ge c-n+\ell + 1,
\label{eq11-lem2-1}
\\
a\le-c+2n- \ell - 1.
\label{eq12-lem2-1}
\end{gather}

Assume that $v_1 = 2n$.  According to~\eqref{eq6-lem2-1}, we have
$$\ell=2n-c-2,$$
so \eqref{eq11-lem2-1} and~\eqref{eq12-lem2-1} yield
$$a\le 1,\ a+\ell\ge n-1.$$
These inequalities and \eqref{eq10-lem2-1} imply that the integer point 
$(c,v_2)$ belongs to~$[\mathbf a, \mathbf b]$.  But then $\m$ contains the 
integer segment $[(c,v_2),\mathbf{v}]$ having $v_1-c+1=\ell+3$ integer points 
(according to~\eqref{eq6-lem2-1}).  This contradicts the definition of the 
lattice diameter.

Now assume $v_1=2n+1$.  From~\eqref{eq6-lem2-1} we get
$$\ell=2n-c - 1,$$
and from \eqref{eq11-lem2-1} and~\eqref{eq12-lem2-1} it follows that
$$a\le0,\ a+\ell \ge n.$$
These inequalities and \eqref{eq10-lem2-1} imply a contradiction exactly as in 
the case $v_1=2n$.

The contradictions show that in fact $\m$ lies in the half-plane $x_1 \le 2n - 
1$.  To show that it lies in the half-plane $x_1\ge - n+1$ as well, it 
suffices to apply the established part of the lemma to the reflection of $\m$ 
about the line $x_1=n/2$.
\end{proof}

\begin{proof}[Proof of Proposition~\ref{pr:slab}]
Let $[\mathbf a, \mathbf b] \in \m $ be a segment containing $\ell(\m) + 1$ 
integer points.  By Proposition~\ref{pr:uni}, there exists a unimodular 
transformation~$A$ such that $A(\mathbf b - \mathbf a) = t\mathbf e_2$, then 
the segment $A[\mathbf a, \mathbf b]$ lies on a line $x_1 = m$, where $m$ is 
an integer.  Let $m=nq+c$, where $q,c\in\z$ and $0\le c\le n-1$ and let $T$ be 
the translation by the vector $-nq\mathbf{e}_1\in\nz$.  Then the segment 
$TA[\mathbf a, \mathbf b]$ contains $\ell(P) + 1 = \ell(TA\m) + 1$ integer 
points and lies on the line $x_1=c$.  In view of Lemma~\ref{lem2-1}, $\psi = 
TA\m$ is the required automorphism.
\end{proof}
\begin{remark}
Let $\m$ be an integer polygon free of $\nz$-points and let $\psi$ be an 
automorphism of~$\nz$ constructed in the proof of Proposition~\ref{pr:slab}.  
As the polygon $\psi(\m)$ does not contain points of the lattice $\nz$, its 
intersection with the line $x_1 = 0$ lies between two adjacent points of 
$\nz$, i.~e.\ it is a subset of a segment $I_1=[(0,u_1n),(0,u_1n+n)]$, where 
$u_1 \in \z$ (if the intersection is empty, $u_1$ can be chosen arbitrarily).  
Likewise, the intersection of $\psi(\m)$ with the line $x_1 = n$ is a subset 
of a segment $I_2=[(n,u_2n),(n,u_2n+n)]$, where $u_2 \in \z$.  Define the 
affine automorphism of $\nz$ by
\begin{equation*}
\tilde \psi(x_1, x_2)
= (x_1, x_2 + (u_1 - u_2) x_1 - u_1 n)
.
\end{equation*}
Since it preserves the first coordinate and maps $I_1$ onto 
$[\mathbf{0},(0,n)]$ and $I_2$ onto $[(n,0),(n,n)]$, it is easily seen that 
the automorphism $\tilde \psi\psi$ satisfies the requirements of 
Proposition~\ref{pr:slab} and Remark~\ref{rem:slab}.
\end{remark}

The following lemma is used in the proof of Theorem~\ref{th:types}.

\begin{figure}
\includegraphics{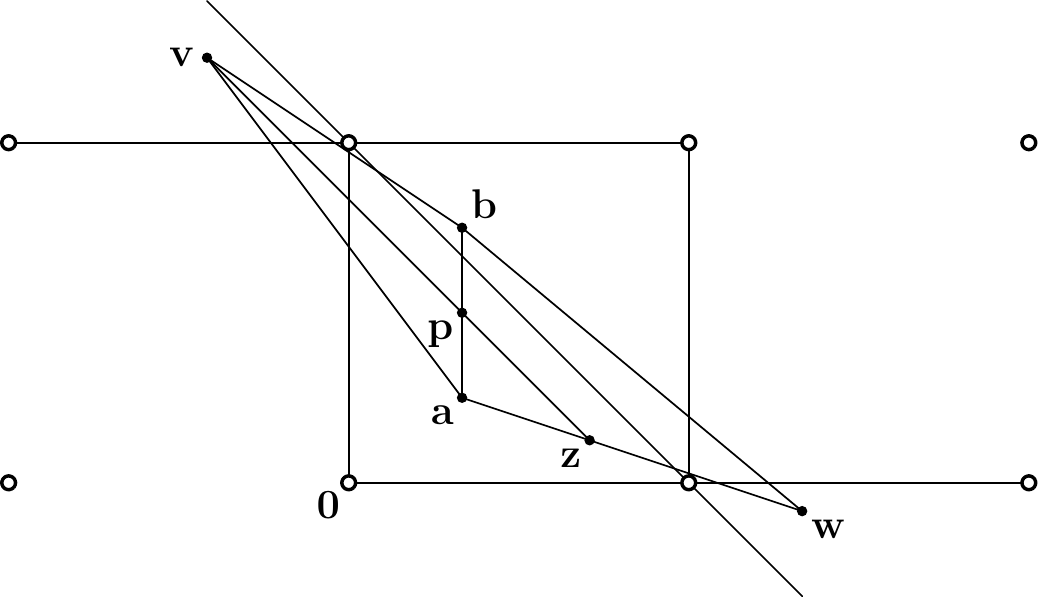}
\caption{Lemma~\ref{lem:imposs}}
\label{fig:imposs}
\end{figure}

\begin{lemma}
\label{lem:imposs}
Under the hypotheses of Lemma~\ref{lem2-1}, the segments $[(0, n), (-n, n)]$ 
and $[(n, 0), (2n, 0)]$ cannot simultaneously split~$\m$.  The same is true 
about the pair of segments $[\mathbf 0, (-n, 0)]$ and $[(n, n), (2n, n)]$.
\end{lemma}
\begin{proof}
By symmetry, it suffices to consider only the first pair of segments.  To 
obtain a contradiction, we assume that both segments split~$\m$.  Then we see 
that $\m$ has vertices $\mathbf v = (v_1, v_2)$ and $\mathbf w = (w_1, w_2)$ 
satisfying
\begin{gather}
v_1 \le -1, v_2 \ge n + 1;
\label{eq:imposs1}
\\
w_1 \ge n+1, \ w_2 \le -1
\label{eq:imposs2}
\end{gather}
(see Figure~\ref{fig:imposs}).

Let $\mathbf a = (c, a)$ and $\mathbf b = (c, b)$ be the endpoints of the 
segment mentioned in the hypothesis, so that $b - a = \ell(P)$.  Observe that
\begin{equation}
\label{eq:imposs4}
a \ge 1, \ b \le n - 1
.
\end{equation}
Indeed, $\m$ cannot have common points with the segments $[\mathbf 0, (n, 0)]$ 
and $[(n, 0), (n, n)]$, for otherwise it would contain an $\nz$-point.  These 
segments are the sides of the square that clearly has common points with~$\m$.  
Then the segment $[\mathbf a, \mathbf b]$ must lie in the slab $0 < x_2 < n$, 
whence~\eqref{eq:imposs4}.

By Lemma~\ref{lem:th2-1} applied to the polygon $\m - \mathbf a$, we have
\begin{gather}
v_1 \ge c - b + a - 2,
\label{eq:imposs3}
\\
w_1 \le c + b - a + 2.
\label{eq:imposs5}
\end{gather}

It follows from~\eqref{eq:imposs3}, \eqref{eq:imposs1}, and~\eqref{eq:imposs4} 
that the the coordinates of $\mathbf a$ satisfy $a + c \le n$, i.~e.~$\mathbf 
a$ cannot lie above the line $x_1 + x_2 = n$.  In the same way, it follows 
from~\eqref{eq:imposs5}, \eqref{eq:imposs2}, and~\eqref{eq:imposs4} that 
$\mathbf b$ cannot lie below this line.  As a consequence, $\mathbf v$ must 
lie below this line, for else~$P$ would have a common point with the segment 
$[(0,n), (n, n)]$; likewise, $\mathbf w$ must lie above this line.

Let $\mathbf p = (c, p)$, where $p = v_1 + v_2 - c$, be the projection of 
$\mathbf v$ on the line $x_1 = c$ along the vector $\mathbf e_1 - \mathbf 
e_2$.  As~$\mathbf v$ lies below the line $x_1 + x_2 = n$, so does~$\mathbf 
p$, and therefore $\mathbf p$ lies below $\mathbf b$.  Moreover, $\mathbf p$ 
cannot lie below $\mathbf a$, since using~\eqref{eq:imposs1} 
and~\eqref{eq:imposs3} and then~\eqref{eq:imposs4}, we have
\begin{equation*}
p = v_1 + v_2 - c \ge a - b + n - 1 \ge a
.
\end{equation*}
Thus, $\mathbf p$ lies on the segment $[\mathbf a, \mathbf b]$.

The points $\mathbf a$, $\mathbf b$, and $\mathbf w$ are the vertices of a 
triangle $T$.  The line $l$ passing through $\mathbf v$ and $\mathbf p$ 
intersects the side $[\mathbf a, \mathbf b]$ of~$T$, so it intersects another 
side as well.  The line $l$ is parallel to the line $x_1 + x_2 = n$ and lies 
below it; on the contrary, neither~$\mathbf b$ nor~$\mathbf w$ lie below the 
latter line, so~$l$ has no common points with the segment $[\mathbf b, \mathbf 
w]$.  Consequently, $l$ intersects the side $[\mathbf a, \mathbf w]$ of~$T$ at 
a possibly non-integer point $\mathbf z = (z_1, z_2)$.  The segment $[\mathbf 
p, \mathbf z]$ is contained in~$T$, so the whole segment $[\mathbf v, \mathbf 
z]$ is contained in~$\m$.

The slope of the line passing through~$\mathbf a$ and~$\mathbf w$ is negative 
(this can be seen by e.~g.\ comparing the coordinates of those points 
using~\eqref{eq:imposs2}), so $z_2 \le a$.  Using~\eqref{eq:imposs1} 
and~\eqref{eq:imposs4}, we can estimage the number of integer points lying on 
the segment $[\mathbf v, \mathbf z]$ as follows:
\begin{equation*}
|[\mathbf v, \mathbf z] \cap \zz| =
\floor{v_2 - z_2} + 1 \ge n - a + 2 \ge b - a + 3 = \ell(P) + 3
,
\end{equation*}
which contradicts the definition of the lattice diameter.  The contradiction 
proves the lemma.
\end{proof}

\section{Properties of slopes}
\label{sec:pfs}

In this section we prove Theorems~\ref{th3-6} and~\ref{th3-8}.
\subsection{Preliminaries}
\label{sec:pfs:prelim}

Throughout this section $(\mathbf o; \mathbf f_1, \mathbf f_2)$ is an integer 
frame splitting an integer slope $Q$.  Let $\mathbf v_0 = \mathbf v$, $\mathbf 
v_1$, \dots, $\mathbf v_N = \mathbf w$ be the vertices of $Q$.  We define 
$\mathbf a_i$ by~\eqref{eq1-2-3} and assume that \eqref{eq3-3} 
and~\eqref{eq4-3} hold.  By $\varepsilon_i = [\mathbf v_{i-1}, \mathbf v_i]$ 
($i = 1, \dots, N$) denote the edges of $Q$ and by $E$, the set of edges.  Set
\begin{gather*}
\mathbf v_i - \mathbf o = v_{i1} \mathbf f_1 + v_{i2} \mathbf f_2
\qquad(i=0,\dots,N)
.
\end{gather*}
Note that $v_{ij}$ and $a_{ij}$ are integers.

Further, set
\begin{gather*}
k=\min\{i\colon v_{i2} < 0\},
\quad
\alpha=-\frac{a_{k1}}{a_{k2}},
\quad
t= \lceil \alpha \rceil - 1,
\\
S=\{\varepsilon_i\in E\colon i<k,\ a_{i2}=-1\},
\quad
s=|S|.
\end{gather*}
All these are well-defined.

\begin{remark}
\label{rem:ssa0}
It is easily seen that $(\mathbf o; \mathbf f_1, \mathbf f_2)$ forms small 
angle with~$Q$ if and only if $\alpha \ge 1$.
\end{remark}

\begin{remark}
\label{rem:ssa}
Let us define $\tilde \alpha = - a_{\tilde k 2}/a_{\tilde k 1}$, where $\tilde 
k = \min\{i \colon v_{k1} \ge 0\}$.  The coefficient~$\tilde \alpha$ is 
related to the frame $(\mathbf o; \mathbf f_2, \mathbf f_1)$ in the same way 
as~$\alpha$ is to $(\mathbf o; \mathbf f_1, \mathbf f_2)$.  The statement of 
Proposition~\ref{pr:sa} saying that at least one of those frames forms small 
angle with~$Q$ can be equivalently expressed in the form of the inequality
\begin{equation*}
\min\{\alpha, \tilde \alpha\} \ge 1
.
\end{equation*}
Moreover, it is not hard to prove that equality holds if and only if $\alpha = 
\tilde \alpha = 1$, in which case $k = \tilde k$, and consequently $v_{k-1,1} 
< 0$ and $v_{k-1,2} > 0$.
\end{remark}

\begin{lemma}
\label{lem:pfs-st}
The cardinality $s$ of $S$ satisfies
\begin{equation}
\label{eq:pfs-st-1}
s \le t
,
\end{equation}
and, moreover,
\begin{equation}
\label{eq:pfs-st-2}
\sum_{\varepsilon_i \in S} a_{i1} \le (t-s)s + \frac{s(s+1)}{2}
.
\end{equation}
If the vertices of $Q$ belong to a proper sublattice of~$\zz$, then $s = t$ 
only if both equal $0$ or $1$, and in the latter case the only edge $\eps_i 
\in S$ has the associated vector $\mathbf a_i = \mathbf f_1 - \mathbf f_2$.
\end{lemma}

\begin{proof}
Let $S = \{\eps_{i_1}, \dots, \eps_{i_s}\}$, where $i_1 < \dots < i_s < k$.  
It follows from~\eqref{eq3-3} and~\eqref{eq4-3} that
\begin{equation*}
\frac{a_{11}}{-a_{12}}
<
\frac{a_{21}}{-a_{22}}
<
\dots
<
\frac{a_{k1}}{-a_{k2}} = \alpha
.
\end{equation*}
Hence, as $a_{i_p2} = -1$, we see that
\begin{equation}
\label{eq:pfs-st-10}
0< a_{i_11} < a_{i_21} < \dots < a_{i_s1} \le \lceil \alpha \rceil - 1 = t
.
\end{equation}
This implies~\eqref{eq:pfs-st-1}.  Moreover,~\eqref{eq:pfs-st-10} implies that 
$a_{i_p1} \le t - s + p$, where $p = 1, \dots, s$, and upon summation, we 
recover~\eqref{eq:pfs-st-2}.  Now suppose that $s = t$, and the vertices of 
$Q$ belong to a proper sublattice $\Gamma$ of $\zz$.  Let us show that either 
$s = 0$ or $s = 1$.  If $s \ne 0$, then $S \ne \varnothing$, and as the 
vectors $\mathbf a_i$ belong to $\Gamma$, we see that the small $\mathbf 
f_2$-step of $\Gamma$ is 1.  By Proposition~\ref{pr:steps}, $\Gamma$ has large 
$\mathbf f_1$-step $m \ge 2$.  The differences $\mathbf a_{i_p} - \mathbf 
a_{i_q} \in \Gamma$ are proportional to $\mathbf f_1$, so the numbers 
$a_{i_p}$, where $p = 1, \dots, s$, differ by multiples of $m$.  Thus, in view 
of~\eqref{eq:pfs-st-10}, we can only have $s = t$ if $s = t = 1$, as claimed.  
In this case $a_{i_11} = 1$, so that $\mathbf a_{i_1} = \mathbf f_1 - \mathbf 
f_2$.
\end{proof}

Given an edge $\varepsilon_i \in E$, define
\begin{gather*}
\pi_1(\varepsilon_i)=v_{i1}^+ - v_{i-1,1}^+, \\
\pi_2(\varepsilon_i)=v_{i-1,2}^+ - v_{i2}^+ , \\
\hat{\pi}(\varepsilon_i)=\pi_1(\varepsilon_i)+\pi_2(\varepsilon_i)-2
\end{gather*}
and extend the functions $\pi_1$, $\pi_2$, and $\hat \pi$ to the set of all 
subsets of~$E$ by additivity.  Observe that these functions take integer 
values.
\begin{remark}
\label{rem:projections}
Obviously, for any $F \subset E$ we have
\begin{gather}
\pi_j(F)\ge 0\qquad (j=1,2),
\label{eq:pi-1}
\\
\hat{\pi}(F)=\pi_1(F) + \pi_2(F)-2|F|
\label{eq:pi-3}
.
\end{gather}
Moreover, it is clear that if $\tilde Q$ is a subslope of $Q$ (not necessarily 
integer) with the endpoints $\tilde{\mathbf v} = \mathbf o + \tilde v_1 
\mathbf f_1 + \tilde v_2 \mathbf f_2$ and $\tilde{\mathbf w} = \mathbf o + 
\tilde w_1 \mathbf f_1 + \tilde w_2 \mathbf f_2$, then
\begin{equation*}
\pi_1(\tilde Q) = | \tilde v_1^+ - \tilde w_1^+ |,
\qquad
\pi_2(\tilde Q) = | \tilde v_2^+ - \tilde w_2^+ |
.
\end{equation*}
\end{remark}

Set
\begin{equation*}
E_1=\{\varepsilon_1, \dots, \varepsilon_k\},
\qquad
E_2=\{\varepsilon_{k+1}, \dots, \varepsilon_N\}
.
\end{equation*}
Clearly, $E_1 \cap E_2 = \varnothing$ and $E_1 \cup E_2 = E$.

\subsection{Auxiliary statements}

\begin{lemma}
\label{lem4-5}
We have
\begin{equation}
\label{eq1-lem4-5}
\hat{\pi}(E_1)
\ge
(v_{k-1, 1}^+ + v_{k-1,2} - 1)
+ \delta
+ (t-s)
+ \floor{(-v_{k2}-1)\alpha}
,
\end{equation}
where
\begin{equation}
\label{eq100-lem4-5}
\delta =
\begin{cases}
1, & \text{if } v_{k-1,2} > 0 \text{ and } \alpha \in \z, \\
0, & \text{otherwise.}
\end{cases}
\end{equation}
\end{lemma}
\begin{proof}
Let us show the inequality
\begin{equation}
\label{eq:pfs-101}
\pi_1(\varepsilon_k) \ge \delta + 1 + t + \lfloor (-v_{k2} - 1) \alpha \rfloor
.
\end{equation}

Since the frame $(\mathbf o; \mathbf f_1, \mathbf f_2)$ splits the slope~$Q$, 
it follows from the definition that the ray $\{\mathbf o + \lambda \mathbf f_1 
\colon \lambda \ge 0\}$ meets~$Q$ at a point $\mathbf z = \mathbf o + z_1 
\mathbf f_1$, where $z_1 > 0$.  As $v_{k2} < 0 \le v_{k-1, 2}$, it is easily 
seen that $\mathbf z$ belongs to the edge $\varepsilon_k$ and either coincides 
with~$\mathbf v_{k-1}$ or is an inner point of the edge.  We consider these 
cases separately.

If $\mathbf z = \mathbf v_{k-1}$, then $v_{k-1, 2} = 0$ and $v_{k-1, 1} > 0$, 
so
\begin{equation*}
\pi_1(\varepsilon_k)
= v_{k1} - v_{k-1,1}
= a_{k1} = \alpha a_{k2} = \alpha(-v_{k2})
= \lceil \alpha \rceil + \lfloor (- v_{k2} - 1) \alpha \rfloor
,
\end{equation*}
and as in this case $\delta = 0$, \eqref{eq:pfs-101} follows.

Assume that $\mathbf z$ is an interior point of $\varepsilon_k$, then $v_{k-1, 
2} > 0$.  As in this case $\mathbf z$ is not the leftmost point of 
$\varepsilon_k$, it is clear that $\pi_1(\varepsilon_k)$ is strictly greater 
than $v_{k1} - z_1 = \alpha(-v_{k2})$.  Thus,
\begin{equation*}
\pi_1(\varepsilon_k) \ge \lfloor \alpha (-v_{k2}) \rfloor + 1
\ge \lfloor \alpha \rfloor + \lfloor (-v_{k2} - 1) \alpha \rfloor + 1
= \lceil \alpha \rceil + \delta + \lfloor (-v_{k2} - 1) \alpha \rfloor
,
\end{equation*}
and \eqref{eq:pfs-101} follows.  The inequality is proved.

Since $v_{k-1, 2} \ge 0$ and $v_{k2} < 0$, we have $\pi_2(\varepsilon_k) = 
v_{k-1, 2}$.  This and~\eqref{eq:pfs-101} implies
\begin{equation}
\label{eq:pfs-102}
\hat \pi(\varepsilon_k) \ge v_{k-1,2} + \delta -1 + t + \lfloor (-v_{k2} - 1) 
\alpha \rfloor
.
\end{equation}

Let $\tilde Q$ be the subslope of $Q$ with the vertices $\mathbf v_0 = \mathbf 
v$, $\mathbf v_1$, \dots, $\mathbf v_{k-1}$ and $\tilde E$ be the set of its 
edges, then $\varepsilon_k \notin \tilde E$, and $\tilde E \cup 
\{\varepsilon_k\} = E_1$.  As $\tilde Q$ lies in the upper half-plane, for any 
$\varepsilon_i \in \tilde E$ we have $\pi_2(\varepsilon_i) = v_{i-1, 2} - 
v_{i2} = - a_{i2}$.  Consequently, $\pi_2(\varepsilon_i) = 1$ if 
$\varepsilon_i \in S$ and $\pi_2(\varepsilon_i) \ge 2$ otherwise, whence
\begin{equation}
\label{eq:pfs-103}
\pi_2(\tilde E) \ge 2|\tilde E| - s
.
\end{equation}
Further, by~Remark~\ref{rem:projections}, we have $\pi_1(\tilde E) = v_{k-1, 
1}^+ - v_{k1}^+ = v_{k-1, 1}^+$.  Combining this with~\eqref{eq:pfs-103} and 
using~\eqref{eq:pi-3}, we obtain
\begin{equation}
\label{eq:pfs-104}
\hat \pi(\tilde E) \ge v_{k-1, 1}^+ - s
.
\end{equation}

As $\hat \pi(E_1) = \hat \pi(\varepsilon_k) + \hat \pi(\tilde E)$, we 
sum~\eqref{eq:pfs-102} and~\eqref{eq:pfs-104} and obtain~\eqref{eq1-lem4-5}.
\end{proof}

\begin{lemma}
\label{lem:pfs-1}
Suppose that $(\mathbf o; \mathbf f_1, \mathbf f_2)$ forms small angle with 
$Q$; then
\begin{equation}
\label{eq:pfs-200}
\hat{\pi}(E_2)\ge
\frac 12(v_{k2} - w_{2} - 1)
.
\end{equation}
\end{lemma}
\begin{proof}
If $E_2 = \varnothing$, we have $\mathbf v_k = \mathbf w$ and 
\eqref{eq:pfs-200} is obvious.

Assume that $E_2\neq\varnothing$ and take $\varepsilon_i\in E_2$.  It is easy 
to see that all the edges belonging to $E_2$ lie in the right half-plane, so 
$\pi_1(\eps_i)=v_{i-1,1}-v_{i1}=a_{i1}$, and since $\pi_2(\eps_i) \ge 0$, we 
obtain
$$\hat{\pi}(\eps_i)\ge a_{i1}-2$$
(actually, the equality holds here).  Write the last inequality in form
\begin{equation}
\label{eq:lem4-7a}
\hat{\pi}(\eps_i)\ge \frac{-a_{i2}+g(a_{i1},a_{i2})}{2},
\end{equation}
where
\begin{equation*}
g(m_1,m_2)= 2m_1 + m_2 - 4
.
\end{equation*}
Using~\eqref{eq4-3}, we see that
$$
\alpha = \frac{a_{k1}}{-a_{k2}}
< \frac{a_{k+1,1}}{-a_{k+1,2}}
< \dots
< \frac{a_{N1}}{-a_{N2}}
;
$$
besides, as $(\mathbf o; \mathbf f_1, \mathbf f_2)$ forms small angle 
with~$Q$, we have $\alpha \ge 1$ (Remark~\ref{rem:ssa0}; consequently,
$$a_{i1} + a_{i2} > 0.$$
Now it is not hard to check that $g$ takes nonnegative values in all the 
points of the set
\begin{equation*}
\{(m_1,m_2)\in\zz \colon m_1>0, m_2<0,m_1+m_2 > 0\}
\end{equation*}
except $(2,-1)$, and $g(2,-1)=-1$.  Consequently,~\eqref{eq:lem4-7a} gives
$$\hat{\pi}(\eps_i)\ge\frac{-a_{i2}-\delta_i}{2},$$
where
\begin{equation*}
\delta_i
=
\begin{cases}
1,&\text{if } \mathbf{a}_i=2\mathbf{f}_1-\mathbf{f}_2,\\
0, & \text{otherwise}.
\end{cases}
\end{equation*}
The vectors $\mathbf a_i$ are distinct, so at most one $\delta_i$ is nonzero.  
Thus, we have:
\begin{equation*}
\hat{\pi}(E_1)
= \sum\limits_{i=k+1}^{N}\hat{\pi}(\eps_i)
\ge \frac12\left(
-\sum\limits_{i=k+1}^{N}a_{i2}
-\sum\limits_{i=k+1}^{N}\delta_i\right)
\ge\frac12(v_{k-1,\,2}-v_{02}-1)
,
\end{equation*}
and~\eqref{eq:pfs-200} is proved.
\end{proof}

\begin{lemma}\label{lem4-7'}
Suppose that $Q$ is a $\Gamma$-slope, where $\Gamma$ is a proper sublattice 
of~$\zz$; then
\begin{equation}
\label{eq1-lem4-7'}
\hat{\pi}(E)\ge1.
\end{equation}
\end{lemma}
\begin{proof}
In view of Remark~\ref{rem:ssa}, there is no loss of generality in assuming 
that $(\mathbf o; \mathbf f_1, \mathbf f_2)$ forms small angle with $Q$, and 
moreover, that $v_{k-1, 2} > 0$ if $\alpha = 1$.

Suppose, contrary to our claim, that $\hat{\pi}(E)\le0$.  As $E = E_1 \cup 
E_2$ and $E_1 \cap E_2 = \varnothing$, we have $\hat \pi(E) = \hat \pi(E_1) + 
\hat \pi(E_2)$.  It follows from Lemma~\ref{lem:pfs-1} that $\hat \pi(E_2) \ge 
0$, so we conclude that $\hat \pi(E_1) \le 0$.  Together with 
Lemma~\ref{lem4-5} this gives
\begin{equation}
\label{eq:pfs-201}
(v_{k-1, 1}^+ + v_{k-1,2} - 1)
+ \delta
+ (t-s)
+ \lfloor (-v_{k2}-1)\alpha \rfloor
\le 0
,
\end{equation}
where $\delta$ is defined by~\eqref{eq100-lem4-5}.  Observe that the summands 
on the left-hand side are nonnegative.  Indeed, the vertex $\mathbf v_{k-1}$ 
cannot simultaneously satisfy $v_{k-1, 1} \le 0$ and $v_{k-1,2} \le 0$ 
(Remark~\ref{rem:slpq}), so the first summand is nonnegative.  The second one 
is nonnegative by definition, the third one by Lemma~\ref{lem:pfs-st}, and the 
last one by the definition of~$k$.  Thus, \eqref{eq:pfs-201} can hold only if
\begin{gather}
v_{k-1, 1}^+ + v_{k-1,2} = 1
,
\label{eq:pfs-202}
\\
\delta = 0
,
\label{eq:pfs-203}
\\
t = s
,
\label{eq:pfs-204}
\\
v_{k2} = -1
\label{eq:pfs-205}
,
\end{gather}
where we recover~\eqref{eq:pfs-205} due to the fact that $\alpha \ge 1$.  
Also, \eqref{eq:pfs-203} implies that $\alpha > 1$, since we are assuming 
$v_{k-1,2} > 0$ if $\alpha = 1$.  Thus, we have $t \ge 1$, and by 
Lemma~\ref{lem:pfs-st} we conclude from~\eqref{eq:pfs-204} that
\begin{equation}
t = s = 1
\label{eq:pfs-206}
\end{equation}
and the set $S$ consists of a single edge $\eps_i$ having the associated 
vector $\mathbf a_i = \mathbf f_1 - \mathbf f_2$.  Thus, we necessarily have 
$\mathbf f_1 - \mathbf f_2 \in \Gamma$.

It follows from~\eqref{eq:pfs-202} that either $v_{k-1,2} = 0$ or $v_{k-1, 2} 
= 1$.

In the former case case we use~\eqref{eq:pfs-205} to get $a_{k2} = v_{k2} - 
v_{k-1,2} = -1$; moreover, $t = 1$ translates into $1 < \alpha \le 2$. Thus, 
$a_{k1} = - \alpha a_{k2} = \alpha$, and consequently, $a_{k1} = 2$.  We 
conclude that $\Gamma$ contains the vector $\mathbf a_k = 2\mathbf f_1 - 
\mathbf f_2$.  But then by Proposition~\ref{pr:th1-1} the vectors $\mathbf 
a_k$ and $\mathbf f_1 - \mathbf f_2$ form a basis of $\zz$,  which is 
impossible, since they belong to its proper sublattice.

In the latter case $v_{k-1, 2} = 1$ we again use~\eqref{eq:pfs-205} to get 
$a_{k2} = -2$.  Moreover, as $t = 1$ and $\alpha \notin \z$ by virtue 
of~\eqref{eq:pfs-203}, we have $1 < \alpha < 2$.  Consequently, the integer 
$a_{k1} = -\alpha a_{k2}$ belongs to the interval~$(2,4)$, i.~e.\ $-a_{k1}=3$.  
Thus, $\mathbf{a}_k=3\mathbf{f}_1-2\mathbf{f}_2$.  But in this case we see 
once again that the vectors~$\mathbf{a}_k$ and $\mathbf{f}_1-\mathbf{f}_2$ 
form a basis of $\zz$, which is impossible since they belong to its proper 
sublattice~$\Gamma$.
\end{proof}

\subsection{Proof of Theorems \ref{th3-6} and \ref{th3-8}}

\begin{proof}[Proof of Theorem~\ref{th3-6}]
\begin{figure}
\includegraphics{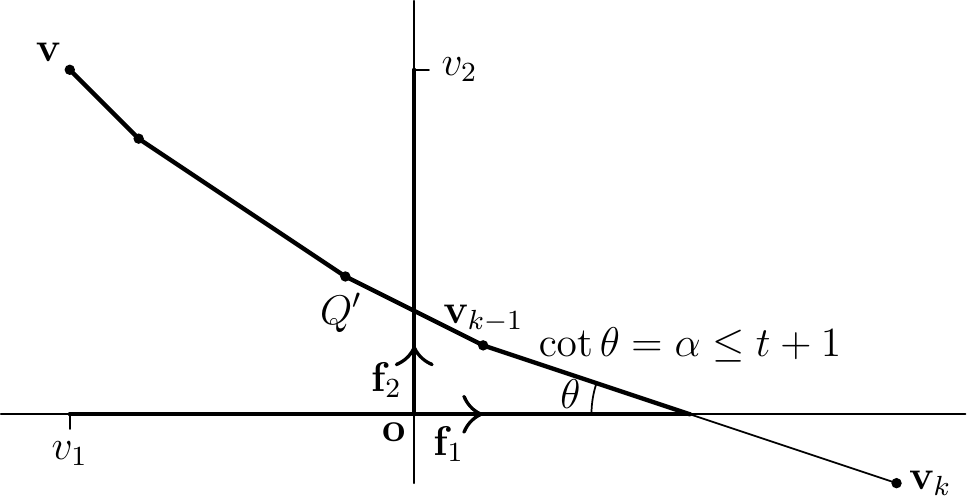}
\caption{The intersection of $Q$ with the upper half-space is a (possibly 
non-integer) slope $Q'$.  The length of the vertical projection of $Q'$ is 
$v_2$ and that of its horizontal projection is strictly greater than $-v_1$.  
All the edges belonging to $S$ are edges of $Q'$.  They all contribute $s$ to 
the length of the vertical projection of $Q'$, whence \eqref{eq:3-6B}.  Also, 
they all contribute at most $t + (t-1) + \dots + (t - s + 1) = ts - s(s-1)/2$ 
to the horizontal projection.  The horizontal contribution of any other edge 
is less than $t + 1$ times its vertical contribution, totalling at most 
$(t+1)(v_2 - s)$ for all the edges not in $S$, and~\eqref{eq:3-6C} follows.}
\label{fig:pfs}
\end{figure}

First assume that the frame $(\mathbf o; \mathbf f_1, \mathbf f_2)$ forms 
small angle with~$Q$, or, equivalently, $\alpha \ge 1$ 
(Remark~\ref{rem:ssa0}).  We show that in these case inequalities 
\eqref{eq:3-6A}--\eqref{eq:3-6E} hold with $t$ and $s$ defined in 
Section~\ref{sec:pfs:prelim}.

Inequalities \eqref{eq:3-6A}--\eqref{eq:3-6C} follow from 
Lemma~\ref{lem:pfs-st} and simple combinatorial arguments, see 
Figure~\ref{fig:pfs}.

Let us prove \eqref{eq:3-6D}.  Since the sets $E_1$ and $E_2$ are disjoint and 
their union is $E$, we have $\hat{\pi}(E)=\hat{\pi}(E_1)+\hat{\pi}(E_2)$.  
Evoking Lemmas~\ref{lem4-5} and~\ref{lem:pfs-1}, we obtain
\begin{multline}
\label{eq:lem4-7c}
\hat{\pi}(E) =
(v_{k-1, 1}^+ + v_{k-1,2} - 1)
+ \delta
+ (t-s)
\\
+ \floor{(-v_{k2}-1)\alpha}
+
\frac 12(v_{k2} - w_{2} - 1)
,
\end{multline}
where $\delta$ is defined by~\eqref{eq100-lem4-5}.  By definition, $\delta \ge 
0$.  The first term on the right-hand side of~\eqref{eq:lem4-7c} is 
nonnegative, since at least one coordinate of $\mathbf v_{k-1}$ must be 
positive (Remark~\ref{rem:slpq}.  Let us estimate the fourth term on the 
right-hand side of~\eqref{eq:lem4-7c}.  By the definition of~$k$ we have 
$v_{k2} < 0$, so $-v_{k2} - 1 \ge 0$, and by assumption, $\alpha \ge 1$; thus, 
we have:
\begin{equation*}
\lfloor (-v_{k2}-1)\alpha \rfloor
\ge
-v_{k2}-1
\ge
\frac 12(-v_{k2}-1)
.
\end{equation*}
Consequently, from~\eqref{eq:lem4-7c} we obtain
\begin{equation*}
\hat{\pi}(E)\ge
t-s+\frac{-w_{2}}{2} - 1.
\end{equation*}
As $\hat \pi(E)$ is an integer, this yields
\begin{equation*}
\hat{\pi}(E)
\ge t-s+ \Ceil{\frac{-w_{2}}{2}} - 1
.
\end{equation*}
As $\pi_1(E) + \pi_2(E) = v_2 + w_1$ by Remark~\ref{rem:projections}, now it 
remains to use~\eqref{eq:pi-3} in order to obtain~\eqref{eq:3-6D}

Now assume that the frame does not form small angle with~$Q$.  Let us check 
that in this case \eqref{eq:3-6A}--\eqref{eq:3-6D} hold with $t = s = 0$.

Inequality~\eqref{eq:3-6A} becomes trivial, and~\eqref{eq:3-6B} follows from 
the definition of a splitting frame.  Inequality~\eqref{eq:3-6C} becomes
$$-v_{N1}<v_{N2}.$$
It is true, since otherwise by Proposition~\ref{pr:sa} the frame would form 
small angle with~$Q$.  To prove~\eqref{eq:3-6D}, it suffices to apply the 
proved part of the theorem to $Q$ and the frame $(\mathbf o; \mathbf f_2, 
\mathbf f_1)$, which forms small angle with $Q$ by Proposition~\ref{pr:sa}.  
Indeed, with certain $\tilde t$ and $\tilde s$ by virtue of~\eqref{eq:3-6D} 
and~\eqref{eq:3-6A} we have:
$$2N\le w_1 + v_2 -\tilde t+ \tilde s \le v_2 + w_1,$$
so \eqref{eq:3-6D} holds for $(\mathbf o; \mathbf f_2, \mathbf f_1)$ with $t = 
s = 0$ as claimed.
\end{proof}

\begin{proof}[Proof of Theorem~\ref{th3-8}]
It suffices to apply Lemma~\ref{lem4-7'}, write inequality~\eqref{eq1-lem4-7'} 
in the form
$$\pi_1(E) + \pi_2(E) -2N\ge1,$$
and substitute $\pi_1(E) = w_1$ and $\pi_2(E) = v_2$ according to 
Remark~\ref{rem:projections}.
\end{proof}

\bibliographystyle{plain}
\bibliography{lit}

\end{document}